\documentclass{amsart} \usepackage{bbm}
\usepackage[top=2.8cm,left=2.8cm,right=2.8cm,bottom=2.6cm]{geometry}
\usepackage{amsfonts}
\usepackage{amsmath,amsthm,amscd,mathrsfs,amssymb,graphicx,enumerate,
  dsfont}
\usepackage{xypic}

\newtheorem{thm}{Theorem}[section]
\newtheorem{cor}[thm]{Corollary}
\newtheorem{lem}[thm]{Lemma}
\newtheorem{prop}[thm]{Proposition}

\newtheorem{defn-lem}[thm]{Definition--Lemma}
\theoremstyle{definition}
\newtheorem{defn}[thm]{Definition}

\newtheorem{rmk}[thm]{Remark}

\newcommand{\C}{\ensuremath\mathds{C}}

\newcommand{\Z}{\ensuremath\mathds{Z}}
\newcommand{\Q}{\ensuremath\mathds{Q}}
\newcommand{\F}{\ensuremath\mathrm{F}}

\newcommand{\PP}{\ensuremath\mathds{P}}
\newcommand{\calO}{\ensuremath\mathcal{O}}

\newcommand{\HH}{\ensuremath\mathrm{H}}
\newcommand{\CH}{\ensuremath\mathrm{CH}}

\begin{document}

\title{Rationality, universal generation and the integral Hodge conjecture}
\author{Mingmin Shen}



\address{
KdV Institute for Mathematics, University of Amsterdam, P.O.Box 94248, 1090 GE Amsterdam, Netherlands}
\email{M.Shen@uva.nl}

\date{\today}

\begin{abstract} 
We use the universal generation of algebraic cycles to relate (stable) rationality to the integral Hodge conjecture. We show that the Chow group of 1-cycles on a cubic hypersurface is universally generated by lines. Applications are mainly in cubic hypersurfaces of low dimensions. For example, we show that if a generic cubic fourfold is stably rational then the Beauville--Bogomolov form on its variety of lines, viewed as an integral Hodge class on the self product of its variety of lines, is algebraic. In dimension $3$ and $5$, we relate stable rationality with the geometry of the associated intermediate Jacobian.
\end{abstract}

\maketitle

\section{Introduction}

An algebraic variety $X$ is \textit{rational} if it contains an open subset that can be identified with an open subset of the projective space of the same dimension. It is called \textit{stably rational} if the product of $X$ and some projective space is rational. The rationality problem is to tell whether a given variety is (stably) rational or not. It is one of the most subtle problems in algebraic geometry. 

We work over the field $\C$ of complex numbers unless otherwise stated. The L\"uroth theorem and Castelnuovo's criterion settled the rationality problem in dimensions one and two. One breakthrough in dimension three was made by Clemens--Griffiths \cite{cg}, where they showed that a smooth cubic threefold is not rational. Other important methods that appeared around the same time include Artin--Mumford \cite{am} and Iskovskikh--Manin \cite{im}. 

The (stable) rationality problem in dimension three is closely related to the geometry of the intermediate Jacobian. 

\begin{thm}\label{thm cgv}
Let $X$ be a smooth projective variety of dimension three and let $(J^3(X),\Theta)$ be its intermediate Jacobian. 
\begin{enumerate}
\item (Voisin \cite{voisin invent}) If $X$ is stably rational, then the mininal class $\frac{\Theta^{g-1}}{(g-1)!}$ is algebraic.
\item (Clemens--Griffiths \cite{cg}) If $X$ is rational, then the minimal class $\frac{\Theta^{g-1}}{(g-1)!}$ is algebraic and effective (which is equivalent to that $J^3(X)$ is a Jacobian of curves).
\end{enumerate}
\end{thm}

The integral Hodge conjecture is the statement that every integral Hodge class is an \textit{algebraic class}, namely the cohomology class of some integral algebraic cycle. It is known that the integral Hodge conjecture is false in general. The relation between the rationality problem and the integral Hodge conjecture is very mysterious. Theorem \ref{thm cgv}, especially the first statement, can be viewed as a beautiful connection between (stable) rationality and the integral Hodge conjecture. In this paper, we develop a method to achieve more such statements. The main applications will be given to cubic threefolds and cubic fourfolds. We first recall the definition of a decomposition of the diagonal.

\begin{defn}[\cite{voisin aj, voisin invent}]\label{defn decomp diag}
Let $X$ be a smooth projective variety of dimension $d$. We say that $X$ admits a \textit{Chow-theoretical decomposition of the diagonal} if
\[
 \Delta_X = X\times x + Z,\quad \text{in } \CH_d(X\times X),
\]
where $x\in X$ is a closed point on $X$ and $Z$ is an algebraic cycle supported on $D\times X$ for some divisor $D\subset X$. We say that $X$ has a \textit{cohomological decomposition of the diagonal} if the above equality holds in $\HH^{2d}(X\times X,\Z)$.
\end{defn}

One important fact is that a stably rational variety always admits a Chow-theoretical (and hence a cohomological) decomposition of the diagonal. Voisin \cite{voisin invent} used decomposition of the diagonal to show new examples of three dimensional unirational varieties which are not stably rational. Then Colliot-Th\'el\`ene--Pirutka \cite{ct-p} generlized this method to show that a very general quartic threefold is not stably rational. Along the same line of ideas, Totaro \cite{totaro} showed that a very general hypersurface of degree in a certain range is not stably rational. Many new results of non-rationality were obtained in the recent years \cite{hkt, hpt1, hpt2, ht, okada}. The smallest possible degree for a hypersurface to be irrational is three. In dimension three, it is not known whether there exists a smooth cubic threefold that is stably rational or that is not stably rational. When it comes the case of smooth cubic fourfolds, we know there exist rational cubic fourfolds \cite{bd,hassett2, ahtv}. It is expected that a very general cubic fourfold is not rational. However, no single cubic fourfold has been proven irrational.

\subsection{Main results} Let $X\subset \PP^{d+1}_{\C}$ be a smooth cubic hypersurface of dimension $d\geq 3$ and let $F=F(X)$ be the variety of lines on $X$. It is known by Altman--Kleiman \cite{ak} that $F$ is a smooth projective variety of dimension $2d-4$. Over $F$ we have the universal family of lines
\[
\xymatrix{ P\ar[r]^q\ar[d]_p &X\\
 F 
}
\] 
Then we can view $P\times P$ as a correspondence from $F\times F$ to $X\times X$. Let $h\in \CH^1(X)$ be the class of a hyperplane section. 

\begin{thm}\label{thm main chow}
Assume that $X$ admits a Chow theoretical decomposition of the diagonal. 
\begin{enumerate}
\item If $d=3$, then there exists a symmetric 1-cycle $\theta$ on $F\times F$ such that
\[
\Delta_X = X\times x + x\times X + \gamma\times h + h\times \gamma + (P\times P)_*\theta, \quad \text{in }\CH^3(X\times X),
\]
for some $\gamma\in \CH_1(X)$.

\item If $d=4$, then there exists a symmetric 2-cycle $\theta$ on $F\times F$ such that
\[
\Delta_X = X\times x + x\times X + \Sigma + (P\times P)_*\theta, \quad \text{in }\CH^4(X\times X),
\]
where $\Sigma\in \CH^2(X)\otimes \CH^2(X)$ is a symmetric decomposable 4-cycle. Moreover if $X$ is very general, then $\Sigma$ can be chosen to be zero.
\end{enumerate}
\end{thm}

When $d=4$, the variety $F$ is a hyperk\"ahler fourfold that is deformation equivalent to the Hilbert scheme of two points on a $K3$-surface. The canonical Beauville--Bogomolov bilinear form 
$$
\mathfrak{B}: \HH^2(F,\Z)\times \HH^2(F,\Z) \longrightarrow \Z
$$
gives rise to an integral Hodge class $q_{\mathfrak{B}}\in \HH^{12}(F\times F,\Z)$. We have the corresponding statement at the level of cohomology as follows.

\begin{thm}\label{thm main coh}
Let $X$ be a smooth cubic hypersurface of dimension $d=3$ or $4$ and let $F$ be the variety of lines on $X$. Then $X$ admits a cohomological decomposition of the diagonal if and only if there exists a symmetric $(d-2)$-cycle $\theta$ on $F\times F$ such that
\[
 [\theta]\cdot \hat\alpha\otimes \hat\beta = \langle \alpha,\beta \rangle_X
\]
for all $\alpha,\beta\in \HH^d(X,\Z)_{\mathrm{tr}}$, where $\hat\alpha := P^*\alpha$. If $d=4$ and $X$ is very general, then the above condition is also equivalent to the fact that the Beauville--Bogomolov form $q_{\mathfrak{B}}$ is algebraic.
\end{thm}

In the case of $d=3$, the above result has the following interesting application.

\begin{cor}
Let $X$ be a smooth cubic threefold and let $(J^3(X),\Theta)$ be its intermediate Jacobian. If $X$ admits a decomposition of the diagonal, then the following statements hold.
\begin{enumerate}
\item The minimal class of $J^3(X)$ is algebraic and supported on a divisor of cohomology class $3\Theta$.
\item Twice of the minimal class of $J^3(X)$ is represented by a symmetric (with respect to the multiplication-by-($-1$) morphism of $J^3(X)$) 1-cycle supported on a theta divisor.
\end{enumerate}
\end{cor}

In the case of cubic fivefolds, our method can also relate rationality with the geometry of the intermediate Jacobian. The price we pay here is that we have to consider the rationality of $X$ and $F$ simultaneously.

\begin{thm}\label{thm cubic fivefold}
Let $X$ be a smooth cubic fivefold and let $F$ be its variety of lines. If both $X$ and $F$ admit a Chow theoretical decomposition of the diagonal, then the intermediate Jacobian $J^5(X)$ is a direct summand of the Jacobian of a (possibly reducible) curve (without respecting the principal polarizations).
\end{thm}

\subsection{Universal generation}
The main ingredient in the proof of the above results is the universal generation of the Chow group of 1-cycles on a cubic hypersurface by lines. This universal generation works over a general base field. Let $Y$ and $Z$ be smooth projective varieties defined over a field $K$. A cycle $\gamma\in \CH^r(Z_{K(Y)})$ is universally generating if a spreading $\Gamma\in \CH^r(Y\times Z)$ induces a universally surjective homomorphism $\Gamma_*: \CH_0(Y)\rightarrow \CH^r(Z)$. This means
\[
 (\Gamma_L)_*: \CH_0(Y_L)\longrightarrow \CH^r(Z_L)
\]
is surjective for all field extensions $L\supset K$. 

\begin{thm}\label{thm main universal generation}
Let $X\subset \PP^{d+1}_K$ be a smooth cubic hypersurface of dimension $d\geq 3$ and let $F$ be its variety of lines. Assume that $\CH_0(F)$ contains an element of degree one. Then the universal line $P\subset F\times X$ restricts to a universally generating 1-cycle $P|_{\eta_F}\in \CH_1(X_{K(F)})$. Namely, $P_*: \CH_0(F) \rightarrow \CH_1(X)$ is universally surjective.
\end{thm}

O.~Benoist pointed out that the condition that $F$ admits a 0-cycle of degree one cannot be removed since the above universal generation fails when $X$ is the universal cubic hypersurface over the generic point of the moduli space.

\subsection{Convention and Notation}

Let $X$ be a smooth projective variety of dimension $d$.
\begin{itemize}
\item $\HH^p(X)$ denotes $\HH^p(X,\Z)$ modulo torsion.

\item If $\alpha_i\in\HH^{k_i}(X)$, $1 \leq i \leq r$, such that $\sum k_i =2d$ then $\alpha_1\cdot \alpha_2 \cdots \cdot \alpha_r$ denotes the intersection number, namely the class $\alpha_1\cup\cdots\cup\alpha_r$ evaluated against the fundamental class $[X]$. In the special case of the middle cohomology, we write
\[
 \langle -,- \rangle_X: \HH^d(X) \times \HH^d(X) \rightarrow \Z,\quad (\alpha,\beta)\mapsto \alpha\cdot \beta.
\]

\item If $\Lambda$ is a Hodge structure, then we use $\mathrm{Hdg}(\Lambda)$ to denote the Hodge classes in $\Lambda$. In the special case $\Lambda = \HH^{2i}(X)$, we use $\mathrm{Hdg}^{2i}(X)$ to denote the Hodge classes in $\HH^{2i}(X)$. The transcendental cohomology $\HH^p(X)_{\mathrm{tr}}$ is the group of all elements $\alpha\in \HH^p(X)$ such that
\[
 \alpha\cdot \beta =0,\quad \text{for all }\beta \in \mathrm{Hdg}^{2d-p}(X).
\]
We use $\HH^{2i}(X)_{\mathrm{alg}} \subseteq \mathrm{Hdg}^{2i}(X)$ to be the subgroup of algebraic classes. The same notation is defined for $\Z$ and $\Q$ coefficients.

\item When $X$ is given an ample class $h\in\HH^2(X,\Z)$, we use $\HH^p(X)_{\mathrm{prim}}$ to denote the associated primitive cohomology.

%\item Let $\Lambda$ be a Hodge structure, we use $\End(\Lambda)$ to denote all the endomorphisms of $\Lambda$ as a Hodge structure, unless otherwise stated.

\item For any $\alpha\in\HH^i(X)$ and $\beta\in\HH^j(Y)$ we use $\alpha\otimes\beta $ denote the element in $\HH^{i+j}(X\times Y)$ which is obtained via the K\"unneth decomposition. Similarly, for $\alpha\in \CH^i(X)$ and $\beta\in \CH^j(X)$, we use $\alpha\otimes\beta$ to denote the decomposable cycle $p_1^*\alpha\cdot p_2^*\beta \in \CH^{i+j}(X\times Y)$.

\item $\eta_X$ denotes the generic point of $X$.

\end{itemize}

\subsection*{Acknowledgement}
This project was started when I was visiting Morningside Center of Mathematics (Beijing); I thank Prof. Ye Tian for the invitation. I also thank Claire Voisin and Johan de Jong for many helpful email communications.

\section{On a filtration on cohomology}

In this section, we study a filtration on cohomology given by Definition \ref{defn filtration}. If a variety $X$ is rational, then $X$ can be obtained from the projective space $\PP^d$ by a successive blow ups and blow downs, with centers of dimension at most $d-2$. Hence the cohomology of $X$ comes from those centers. In this case the filtration corresponds to the dimensions of the centers.

In the first subsection, we give some general facts about the filtration including the behavior under a smooth blow-up (Proposition \ref{prop fil blow up}). In the second subsection, we treat an important feature of the filtration, namely that some algebraic cycle of $X$ over a function field gives rise to a bilinear pairing on certain piece of the filtration. %By a result of Voisin \cite{voisin universal}, such bilinear pairings are closely related to the existence of a cohomological decomposition of the diagonal.

Let $X$ be a smooth projective variety of dimension $d$.  Recall that we use $\HH^p(-)$ to denote the integral cohomology $\HH^p(-,\Z)$ modulo torsion.

\begin{defn}\label{defn filtration}
We define an increasing filtration on the middle cohomology $\HH^d(X,\Z)$ by
\begin{equation}\label{eq def fil}
F^i\HH^d(X,\Z) := \bigcap_{(Y,\Gamma)} \ker\{[\Gamma]^*: \HH^d(X,\Z) \rightarrow \HH^*(Y,\Z)\},\quad i\geq 1,
\end{equation}
where $(Y,\Gamma)$ runs through all smooth projective varieties $Y$ of dimension at most $d-i$ and all correspondences $\Gamma\in\CH_*(Y\times X)$. We can similarly define the corresponding notion on $\HH^d(X)$ and $\HH^d(X,\Q)$.
\end{defn}

The following lemma shows that we can actually put some extra restriction on the dimension of $\Gamma$ in the above definition.
\begin{lem}
The filtration $F^i\HH^d(X,\Z)$ can also be defined as in \eqref{eq def fil}, where $(Y,\Gamma)$ runs through all smooth projective varieties $Y$ of dimension at most $d-i$ and all correspondences $\Gamma\in\CH_l(Y\times X)$ with $l\leq d-1$.
\end{lem}

\begin{proof}
Let $F'^i\HH^d(X,\Z)$ be the filtration as defined in the lemma. Then it is clear that $F^i\subseteq F'^i$. Let $\alpha\in F'^i\HH^d(X,\Z)$ and let $Y$ be a smooth projective variety of dimension at most $d-i$ and $\Gamma\in \CH_l(Y\times X)$. We need to show that $\beta:=[\Gamma]^*\alpha = 0$ in $\HH^{d-2(l-d_Y)}(Y,\Z)$. For this, we may further assume that $\Gamma$ is represented by an irreducible subvariety of $Y\times X$ which dominates $Y$. Indeed, if this case is proved and we have a case where $\Gamma$ maps to a proper closed subvariety $j: Z\hookrightarrow Y$, then the action of $\Gamma$ factors through a resolution $Z'$ of $Z$. To be more precise, $[\Gamma]^* = j'_*\circ [\Gamma']^*$ where $\Gamma'$ is a correspondence between $Z'$ and $X$ which dominates $Z'$ and $j': Z'\rightarrow X$ is the resolution morphism $Z'\rightarrow Z$ followed by the inclusion $j$. By assumption, we have $[\Gamma']^*\alpha=0$ and it follows that $[\Gamma]^*\alpha = 0$.

Now we assume that $\Gamma$ dominates $Y$ and hence $l\geq d_Y$. If $l\leq d-1$, then by the definition of $F'^i$ we have $[\Gamma]^*\alpha = 0$. Assume that $l\geq d$. Note that $d':=d-2(l-d_Y) = d_Y -(l-d)-(l-d_Y) \leq d_Y$. The eqality holds only when $l=d=d_Y$ which forces $i=0$. Hence we have $d'<d_Y$. Thus we can take a general complete intersecion $Y'\subset Y$ of very ample divisors in $Y$ such that $\dim Y' = d'$. The Lefschetz hyperplane theorem implies that $\beta = 0$ if and only if $\beta|_{Y'} = 0 $. But we have
$\beta|_{Y'} = [\Gamma']^*\alpha$ where $\Gamma' = \Gamma|_{Y'\times X} \in \CH_{l'}(Y'\times X)$ with $l' = l-(d_Y-d') = d-(l-d_Y)\leq d$. If this is a strict inequality, then by assumption, we have $\beta|_{Y'}=0$ and hence $\beta = 0$. Otherwise, we have $l=d_Y\geq d$ which again forces $i=0$.
\end{proof}

From the above lemma, we see that $F^1\HH^d(X,\Z)$ consists of all $\alpha\in\HH^d(X,\Z)$ such that $f^*\alpha = 0$ in $\HH^d(Y,\Z)$ for all morphisms $f:Y\rightarrow X$ where $\dim Y \leq d-1$. Let $\alpha\in \HH^{d,0}(X)$ be the class represented by a global holomorphic $d$-form, then $f^*\alpha =0$ for all $f: Y\rightarrow X$ with $\dim Y \leq d-1$, since $\HH^{d,0}(Y)=0$ for dimension reasons.

\subsection{The filtration under a blow up}
In this subsection we take $X$ to be a smooth projective variety of dimension $d$. 

\begin{prop}
Let $X'$ be another smooth projective variety of dimension $d$ and $\Gamma\in\CH_{d}(X' \times X)$, then $[\Gamma]$ induces
\[
 [\Gamma]^*: F^k\HH^d(X,\Z) \longrightarrow F^k\HH^d(X',\Z).
\]
In particular, for any morphism $f:X'\rightarrow X$, the homomorphism $f^*$ on middle cohomology respects the filtration.
\end{prop}
\begin{proof}
This can be checked directly using composition of correspondences.
\end{proof}

\begin{prop}\label{prop fil blow up}
Let $\rho:\tilde{X}\rightarrow X$ be a blow up along a smooth center $Y\subset X$ of codimension $r$. Then
\[
\rho^*: F^i\HH^d(X,\Z) \rightarrow F^i\HH^d(\tilde{X},\Z)
\]
is an isomorphism for all $i\leq r$. In particular, the groups $F^1\HH^d(X,\Z)$ and $F^2\HH^d(X,\Z)$ are birational invariants, and they vanish if $X$ is rational. 
\end{prop}

\begin{proof}
Consider the blow up diagram
\[
\xymatrix{
 E\ar[r]^j\ar[d]_\pi & \tilde{X}\ar[d]^\rho \\
 Y\ar[r]^{i_Y} &X
}
\]
where $E\cong \PP(\mathcal{N}_{Y/X})$ is the exceptional divisor. Then it is known that
\begin{equation}\label{eq blow up formula}
 \HH^d(X,\Z)\oplus \left(\bigoplus_{l=1}^{r-1}\HH^{d-2l}(Y,\Z) \right)\cong \HH^d(\tilde{X},\Z),\qquad (\alpha,\beta_1,\ldots,\beta_{r-1})\mapsto \rho^*\alpha +\sum_{l=1}^{r-1}j_*(\xi^{l-1}\cup \pi^*\beta_l),
\end{equation}
where $\xi\in\HH^2(E,\Z)$ is the class of the relative $\calO(1)$-bundle on $E$. Since $\rho_*\rho^*=\mathrm{id}$, we have
\begin{equation}\label{eq inj}
 \rho^*: F^i\HH^d(X,\Z)\hookrightarrow F^i\HH^d(\tilde{X},\Z).
\end{equation}
Now let 
\[
\tilde{\alpha}=\rho^*\alpha + \sum_{l=1}^{r-1}j_*(\xi^{l-1}\cup\pi^*\beta_l) 
\]
be an element in $F^i\HH^d(\tilde{X},\Z)$. Then let 
\[
\Gamma_{i'}:=(\pi,j)_*\xi^{i'-1} \in \CH_{d-i'}(Y\times \tilde{X}), \quad 1\leq i'\leq r-1,
\]
where $(\pi,j): E\rightarrow Y\times X$ is the natural morphism. Since $\dim Y = d-r \leq d-i$, we have
\[
0=\Gamma_1^*\tilde{\alpha} = \pi_*j^*\tilde{\alpha} = -\beta_{r-1}.
\]
Then we apply $\Gamma_2^*$ to $\tilde{\alpha}$ and get $\beta_{r-2}=0$. By induction, we get $\beta_l=0$ for all $l=1,\ldots, r-1$. To show that $\alpha\in F^i\HH^d(X,\Z)$, pick an arbitrary variety $Z$ of dimension $d-i$ and a correspondence $\Gamma\in\CH_{*}(Z\times X)$. Then 
\[
\Gamma^*\alpha = \Gamma^*\rho_*\tilde{\alpha}  = ({}^t\rho \circ \Gamma)^*\tilde{\alpha} =0
\]
and hence $\alpha\in F^i\HH^d(X,\Z)$. It follows that $\rho^*$ in equation \eqref{eq inj} is also surjective.
\end{proof}

\begin{prop}\label{prop fil cubic}
Let $X$ be a smooth cubic hypersurface of dimension $d=3$ or $4$. Then 
\[
F^1\HH^d(X,\Z)= F^2\HH^d(X,\Z)=0,\quad F^3\HH^d(X,\Z) = \HH^d(X,\Z)_{\mathrm{tr}}.
\]
\end{prop}
\begin{proof}
We first note that the middle cohomology of $X$ is torsion free. Let $F$ be the variety of lines on $X$. Let $l\subset X$ be a general line on $X$ and let $D_l\subseteq F$ be the variety of lines that meet $l$. Then it is known that $D_l$ is smooth (see \cite[Lemma 10.5]{cg} for $d=3$; \cite[\S3 Lemme 1]{voisin 86} for $d=4$). The associated Abel--Jacobi homomorphism
\[
\Phi_l: \HH^d(X,\Z)\rightarrow \HH^{d-2}(D_l,\Z)
\]
is injective. This can be seen from the intersection property that
\[
\Phi_l(\alpha)\cdot \Phi_l(\beta) = [D_l]\cdot \hat\alpha \cdot \hat\beta = -2 \langle \alpha,\beta \rangle_X
\]
for all $\alpha,\beta\in \HH^d(X,\Z)_{\mathrm{prim}}$; see \cite{cg,pt}. As a consequence, we see that
\[
 F^2\HH^d(X,\Z) = F^1\HH^d(X,\Z) = 0.
\]
Let $\alpha\in F^3\HH^d(X,\Z)$. If $d=3$, then we have $\HH^5(X,\Z)=0$ and $\HH^7(X,\Z)=0$. As a consequence, we have $F^3\HH^3(X,\Z) = \HH^3(X,\Z)_{\mathrm{tr}} = \HH^3(X,\Z)$. If $d=4$, then for all curves $C$ and correspondences $\Sigma\in \CH_l(C\times X)$, $l\leq 3$, we have $\Sigma^*\alpha = 0$ in $\HH^{6-2l}(C,\Z)$. This condition is nontrivial only when $l=2$ or $3$. In either case, it is equivalent to $[Z]\cdot \alpha = 0$ for some surface $Z\subset X$. So we have $F^3\HH^4(X,\Z)= \HH^4(X,\Z)_{\mathrm{tr}}$.
\end{proof}

\subsection{Bilinear form associated to a cycle over a function field}

\begin{defn-lem}\label{defn-lem bilinear form}
Let $K$ be the function field of a variety of dimension $d-2r$. For any $\gamma\in\CH_r(X_K)$  we can define a $(-1)^d$-symmetric bilinear form
\[
 \langle - ,- \rangle_{\gamma}: F^{2r+1}\HH^d(X) \times F^{2r+1}\HH^d(X) \rightarrow \Z
\]
by $(\alpha,\beta)\mapsto \langle \Gamma^*\alpha,\Gamma^*\beta \rangle_Z$, where $Z$ is a model of $K$ and $\Gamma\subset Z\times X$ is a spreading of $\gamma$. 
\end{defn-lem}

\begin{proof}
The above bilinear form is independent of the choices made. Indeed, a different choice of $Z$ and $\Gamma$  gives rise to an action on $\HH^d(X)$ that differs by the action of a correspondence that factors through a variety of dimension at most $d-2r-1$. Hence the difference action is zero on $F^{2r+1}\HH^d(X)$ by definition.
\end{proof}

\begin{prop}\label{prop pairing torsion}
If $\gamma_1,\gamma_2\in\CH_r(X_K)$ such that $\gamma_1-\gamma_2$ is torsion in $\CH_r(X_K)$, then $\langle - , - \rangle_{\gamma_1} = \langle -,-\rangle_{\gamma_2}$.
\end{prop}
\begin{proof}
By definition $n(\gamma_1-\gamma_2)=0$ in $\CH_r(X_K)$. If we take some spreading $\Gamma_i\in\CH_{d-r}(Z\times X)$ of $\gamma_i$, we see that $n(\Gamma_1-\Gamma_2)$ is supported over a proper closed subset of $Z$. Thus the action of $n(\Gamma_1-\Gamma_2)$ factors through varieties of dimension at most $d-2r-1$. Hence, by definition, we have
\[
[ n(\Gamma_1-\Gamma_2)]^*\alpha = 0,\quad\text{for all }\alpha\in F^{2r+1}\HH^d(X).
\]
It follows that $\Gamma_1^*\alpha = \Gamma_2^*\alpha$ since $\HH^d(X)$, namely $\HH^d(X,\Z)$ modulo torsion, is torsion free. 
\end{proof}

\section{Universal generation}
In this section we give the definition of universal generation of algebraic cycles and discuss its basic properties. Then we discuss how universal generation is related to the decomposition of the diagonal.

We will eventually work over $\C$; but for the definitions, we assume that $X$ is a smooth projective variety of dimension $d$ over an arbitrary field $k$. Let $Z$ be a smooth projective variety with function field $K=k(Z)$. For any cycle $\gamma\in \CH_r(X_K)$, we can define
\begin{equation}\label{eq gamma on CH0}
 \gamma_*: \CH_0(Z)\rightarrow \CH_r(X),\qquad \tau\mapsto\Gamma_*\tau,
\end{equation}
where $\Gamma\in\CH_{d_Z+r}(Z\times X)$ is a spreading of $\gamma$. Namely, $\Gamma|_{\eta_Z\times X} = \gamma$ in $\CH_r(X_K)$. If $\Gamma'$ is another spreading of $\gamma$, then $\Gamma-\Gamma'$ is supported on $D\times X$ for some divisor $D$ of $Z$. It follows that $\Gamma_*\tau = \Gamma'_*\tau$ for all $\tau\in \CH_0(Z)$. Thus the homomorphism \eqref{eq gamma on CH0} only depends on the class $\gamma$. Furthermore, for every field extension $L\supset k$, we have the induced homomorphism
\begin{equation}\label{eq gamma on CH0 L}
(\gamma_L)_*: \CH_0(Z_L)\longrightarrow \CH_r(X_L),\qquad, \tau\mapsto (\Gamma_L)_*\tau.
\end{equation}
The above construction can be generalized to the following situation. Let $Z$ be the disjoint union of smooth projective varieties $Z_i$ with function field $K_i$ and $\gamma = \sum_i\gamma_i$, where $\gamma_i\in\CH_r(X_{K_i})$. Then we can again define
\begin{equation}\label{eq gamma on CH0 L 2}
 (\gamma_L)_*: \bigoplus_i \CH_0((Z_i)_L) \longrightarrow \CH_r(X_L).
\end{equation}

\begin{defn}\label{defn universal generation}
The cycle 
$$
\gamma = \sum_{i=1}^n \gamma_i \in \bigoplus_{i=1}^n \CH_r(X_{K_i}),
$$
is \textit{universally generating} if the natural homomorphism \eqref{eq gamma on CH0 L 2} is surjective for all field extensions $L\supset k$. We say that $\CH_r(X)$ is \textit{universally trivial} if the natural homomorphism $\CH_r(X) \rightarrow \CH_r(X_L)$ is an isomorphism for all field extensions $L\supset k$.
\end{defn}

For a $d$-dimensional variety $X$ with a $k$-point $x$, the universal triviality of $\CH_0(X)$ is equivalent to the existence of a Chow-theoretic decomposition of the diagonal of the form $\Delta_X = X\times x + \Gamma$, in $\CH_d(X\times X)$, where $\Gamma$ is supported on $D\times X$ for some divisor $D\subset X$. Indeed, the existence of a decomposition of the diagonal implies that $\CH_0(X)$ is universally generated by the point $x$. Conversely, if $\CH_0(X)$ is universally trivial in the above sense, then there exists $\gamma\in \CH_0(X)$ such that
\[
\delta_X:=\Delta_X|_{\eta_X\times X} = \gamma,\quad \text{in }\CH_0(X_K),
\]
where $K=k(X)$ is the function field of $X$. This implies that 
\[
\Delta_X = X\times\gamma + \Gamma,\quad \text{in }\CH_d(X\times X),
\]
where $\Gamma$ is supported on $D\times X$ for some divisor $D\subset X$. We apply the above correspondence to $x\in \CH_0(X)$ and get
\[
x = (\Delta_X)_*x = (X\times \gamma)_*x = \gamma.
\]
Hence we have $\Delta_X=X\times x +\Gamma$ in $\CH_d(X\times X)$, which is a decomposition of the diagonal. Thus the universal triviality of $\CH_0$ as in Definition \ref{defn universal generation} agrees with the usual definition; see for example \cite{voisin universal}.

\subsection{Universal generation from decomposition of diagonal}

\begin{prop}[Voisin \cite{voisin universal}]\label{prop voisin}
Let $X$ be a smooth projective variety of dimension $d$ over $\C$. Then $X$ admits a Chow-theoretical decomposition of the diagonal if and only if the following condition ($\ast$) holds.

($\ast$) There exist smooth projective varieties $Z_i$ of dimension $d-2$, correspondences $\Gamma_i\in \CH^{d-1}(Z_i\times X)$ and integers $n_i\in \Z$, $i=1,2,\ldots,r$, such that
\[
 \Delta_X  = \sum_{i=1}^r n_i \Gamma_i\circ{}^t\Gamma_i + X\times x + x\times X, \qquad \text{in }\CH^d(X\times X),
\]
where $x\in X$ is a closed point.
\end{prop}
\begin{proof}
The proof is the same as that of \cite[Theorem 3.1]{voisin universal}, with homological equivalence replaced by rational equivalence. We only sketch the main steps here. It is clear that the condition ($\ast$) is a special form of Chow-theoretical decomposition of the diagonal. For the converse we assume that $X$ has a Chow-theoretical decomposition of the diagonal
\[
 \Delta_X - X\times x = Z, \qquad \text{in }\CH_d(X\times X),
\]
where $Z$ is supported on $D\times X$ for some divisor $D\subset X$. We may relace $X$ by a blow-up and assume that $D=\cup D_i$ is a global normal crossing divisor. Let $k_i: D_i\longrightarrow X$ be the inclusion map. Then there exist $\Gamma'_i\in \CH_{d}(D_i\times X)$ such that
\[
 \Delta_X - X\times x = \sum (k_i,Id_X)_*\Gamma'_i = \sum \Gamma'_i \circ k_i^*.
\]
Composing the above equation with its transpose, we get
\begin{align*}
\Delta_X -X\times x - x\times X & = (\Delta_X - X\times x)\circ(\Delta_X - x\times X)\\
 & = \sum_{i,j} \Gamma'_i\circ k_i^*\circ k_{j,*}\circ {}^t\Gamma'_j.
\end{align*}
For each $i$, assume that $k_i^* D_i = \sum_l n_{i,l}Z'_{i,l}$, where $Z'_{i,l}\subset D_i$ are irreducible divisors. Let $Z_{i,l}$ be a resolution of $Z'_{i,l}$ and let $\Gamma_{i,l}\in \CH_{d-1}(Z_{i,l}\times X)$ be the restriction of $\Gamma'_i$ to $Z_{i,l}$. then we have
\[
 \Gamma'_i\circ k_i^*\circ k_{i,*} \circ {}^t\Gamma'_i = \sum_l n_{i,l} \Gamma_{i,l}\circ {}^t\Gamma_{i,l}.
\]
For $i\neq j$, we set $Z_{\{i,j\}}$ to be the intersection of $D_i$ and $D_j$. Let $\Gamma'_{i,j}\in \CH_{d-1}(Z_{\{i,j\}} \times X)$ be the restriction of $\Gamma'_i$ to $Z_{\{i,j\}}$. Thus we have
\[
\Gamma'_i\circ k_i^*\circ k_{j,*}\circ {}^t\Gamma'_j = \Gamma'_{i,j}\circ {}^t\Gamma'_{j,i}.
\]
It follows that
\[
 \Gamma'_i\circ k_i^*\circ k_{j,*}\circ {}^t\Gamma'_j + \Gamma'_j\circ k_j^*\circ k_{i,*}\circ {}^t\Gamma'_i = (\Gamma'_{i,j}+\Gamma'_{j,i}) \circ ({}^t\Gamma'_{i,j}+ {}^t\Gamma'_{j,i}) - \Gamma'_{i,j}\circ {}^t  \Gamma'_{i,j} - \Gamma'_{j,i}\circ {}^t  \Gamma'_{j,i}
\]
Hence $\Delta_X - X\times x - x \times X$ is of the given form.
\end{proof}

%\begin{cor}\label{cor to voisin}
%Then there exist smooth projective varieties $Z_i$ of dimension $d-2$ with function field $K_i$, 1-cycles $\gamma_i\in\CH_1(X_{K_i})$ and itegers $n_i\in \Z$, $i=1,2,\ldots,n$, such that the following conditions holds.
%\begin{enumerate}[(i)]
%\item We have
%\[
% \langle\alpha,\beta \rangle_X = \sum n_i \langle\alpha,\beta \rangle_{\gamma_i}
%\]
%for all $\alpha,\beta\in F^3\HH^d(X,\Z)$.

%\item The cycle $\gamma = \sum n_i\gamma_i$ (and hence also $\gamma'=\sum \gamma_i$) generates $\CH_1(X)$ universally, modulo $(d-3)$-dimensional variation.
%\end{enumerate}
%\end{cor}

%\begin{proof}
%The cycles $\Gamma_i$ in the proposition restricts to $\gamma_i\in \CH_1(X_{K_i})$, which satisfies the two conditions.
%\end{proof}

The above result of Voisin is sufficient for our main application to cubic threefolds and cubic fourfolds. However, there is a more general version which we need for a later application to cubic fivefolds.

\begin{prop}\label{prop voisin generalized}
Let $X$ be a smooth projective variety of dimension $d$ over $\C$. Let $r\geq 1$ be an integer such that $2r\leq d$. Assume that $\CH_i(X)$ is universally trivial for $i=0,1,\ldots,r-1$. Then we have the following higher decomposition of diagonal
\[
\Delta_X = (X\times x + x\times X)+ \Sigma + \Gamma,\quad \in \CH_d(X\times X),
\]
which satisfies the following conditions.
\begin{enumerate}
\item The cycle 
$$\Sigma\in \bigoplus_{i<r}\left( \CH^i(X)\otimes \CH_i(X) \oplus \CH_i(X)\otimes\CH^i(X) \right)$$ 
is a decomposable symmetric $d$-cycle, namely a linear combination of cycles of the form $\gamma\otimes \gamma' + \gamma'\otimes \gamma$.

\item There are smooth projective varieties $Z_i$ of dimension $d-2r$, correspondences $\Gamma_i\in \CH_{d-r}(Z_i\times X)$ and integers $n_i$ such that
\[
 \Gamma = \sum_i n_i \Gamma_i\circ \sigma_i \circ{}^t\Gamma_i, \quad \text{in }\CH_d(X\times X),
\]
where $\sigma_i: Z_i \rightarrow Z_i$ is a (possibly trivial) involution. If $r=1$, then all $\sigma_i$ can be taken to be identity.
\end{enumerate}
\end{prop}

The proof requires the following lemma concerning symmetric cycles.

\begin{lem}\label{lem symmetric}
Let $X$ be a smooth projective variety. The following statements are true.
\begin{enumerate}[(i)]
\item Let $\Gamma_1$ and $\Gamma_2$ be two symmetric cycles on $X\times X$, then $\Gamma_1\cdot \Gamma_2$ is represented by a symmetric cycle.
\item Let $Z$ be another smooth projective varieties and let $\Gamma\in \CH^p(Z\times X)$ be a correspondence. If a self correspondence $\Sigma \in \CH^q(X\times X)$ is represented by a symmetric cycle, then the self correspondence ${}^t\Gamma\circ\Sigma\circ \Gamma \in \CH^*(Z\times Z)$ is also represented by a symmetric cycle.
\end{enumerate}
\end{lem}

\begin{proof}
Let $d=\dim X$ and let $\rho_X: \widetilde{X\times X} \rightarrow X\times X$ be the blow up along the diagonal. Let $\sigma_X: \widetilde{X\times X} \rightarrow X^{[2]}$ be ramified degree 2 cover of the Hilbert square of $X$. Let $\mu_X=\sigma_{X,*}\rho_X^* \in \CH^{2d}(X\times X \times X^{[2]})$ be the associated correspondence. Let $\iota_X:X\rightarrow X\times X$ be the diagonal embedding. Our first observation is that a cycle class $\gamma\in \CH_r(X\times X)$ is represented by a symmetric cycle if and only if $\gamma = \iota_{X,*}\alpha + \mu_X^*\beta$ for some $\alpha\in \CH_r(X)$ and $\beta\in \CH_r(X^{[2]})$. This can be seen as follows. If $\gamma = \iota_{X,*}\alpha + \mu_X^*\beta$ then it is clear that $\gamma$ is represented by a symmetric cycle. Conversely, assume that $\gamma$ is represented by a symmetric cycle $\Gamma$ on $X\times X$. Let $\Gamma_0$ be the restriction of $\Gamma$ to $X\times X\backslash \Delta_X$. Then there exists a cycle $\Gamma'_0$ on $X^{(2)}\backslash\Delta_X$ whose pull-back to $X\times X\backslash \Delta_X$ is $\Gamma_0$. Note that $X^{(2)}\backslash \Delta_X\subset X^{[2]}$ is canonically an open subvariety. Let $\overline{\Gamma'}_0$ be the closure of $\Gamma'_0$ in $X^{[2]}$. Let $\beta\in \CH_r(X^{[2]})$ be the class of $\overline{\Gamma'}_0$. By construction $\gamma$ and $\mu_X^*\beta$ agrees on $X\times X\backslash \Delta_X$. Hence by the localization sequence, there exists $\alpha\in \CH_r(X)$ such that $\gamma = \iota_{X,*}\alpha + \mu_X^*\beta$.

To prove (i), we write $\Gamma_i = \iota_{X,*}\alpha_i + \mu_X^*\beta_i$, $i=1,2$. Then we have
\begin{align*}
\Gamma_1\cdot \Gamma_2 & = (\iota_{X,*}\alpha_1 + \mu_X^*\beta_1) \cdot (\iota_{X,*}\alpha_2 + \mu_X^*\beta_2)\\
& = \mu_X^*\beta_1 \cdot \mu_X^*\beta_2 \qquad \mod \;\; \iota_{X,*}\CH_*(X)\\
& = \rho_{X,*}\sigma_X^*\beta_1 \cdot \rho_{X,*}\sigma_X^*\beta_2 \qquad \mod \;\; \iota_{X,*}\CH_*(X)\\
& = \rho_{X,*}(\sigma_X^*\beta_1 \cdot \rho_X^*\rho_{X,*}\sigma_X^*\beta_2) \qquad \mod \;\; \iota_{X,*}\CH_*(X)\\
& = \rho_{X,*}\Big( \sigma_X^*\beta_1 \cdot (\sigma_X^*\beta_2 + \tau) \Big) \qquad \mod \;\; \iota_{X,*}\CH_*(X)
\end{align*}
where $\tau$ is supported on the exceptional divisor $E_X\subset \widetilde{X\times X}$ of the blow up $\rho_X$. Thus $\rho_{X,*}(\tau\cdot \sigma_X^*\beta_1)$ is supported on $\Delta_X$. As a consequence, we have
\[
\Gamma_1\cdot \Gamma_2 = \mu_X^*(\beta_1\cdot\beta_2) + \iota_{X,*}\alpha
\]
for some $\alpha\in \CH_*(X)$. It follows that $\Gamma_1\cdot\Gamma_2$ is represented by a symmetric cycle.

Now we prove (ii). By definition
\[
{}^t\Gamma\circ\Sigma\circ \Gamma = (p_{Z\times Z})_*\Big( (p_{X\times X})^*\Sigma \cdot (\Gamma\times \Gamma) \Big).
\]
Note that $(p_{X\times X})^*\Sigma$ is a symmetric cycle on $Z\times X\times Z \times X$ and so is $\Gamma\times \Gamma$. We apply (i) and conclude that ${}^t\Gamma\circ\Sigma\circ \Gamma$ is the push-forward of a symmetric cycle on $Z\times X\times Z\times X$ which is again symmetric.
\end{proof}

\begin{proof}[Proof of Proposition \ref{prop voisin generalized}]
Since $\CH_0(X)$ is universally trivial, we have a decomposition of the diagonal
\[
 \Delta_X = X \times x + \Gamma', \quad\text{in }\CH_d(X\times X),
\]
where $\Gamma'$ is supported on $D\times X$ for some divisor $D$. We assume that $D=\cup_j D_j$. Then the restriction of $\Gamma'$ to the generic point of $D_j$ is a 1-cycle on $X_{\C(D_j)}$. By the universal triviality of $\CH_1(X)$, there exist a 1-cycle $\gamma_j\in \CH_1(X)$ such that $D_j\otimes \gamma_j$ agrees with $\Gamma'$ over the generic point of $D_j$. Thus we have
\[
\Delta_X = X\times x + \Gamma'_1 + \Gamma'',\quad\text{in }\CH_d(X\times X),
\]
where $\Gamma'_1= \sum_j D_j \otimes \gamma_j$ and $\Gamma''$ is supported on $Y\times X$, where $Y\subset X$ is a closed subset of codimension 2. By repeating the above argument, we eventually get
\[
 \Delta_X = X\times x + \Gamma'_1 + \cdots + \Gamma'_{r-1} + \Gamma''',\quad \text{in }\CH_d(X\times X),
\]
where $\Gamma'_i\in \CH^i(X)\otimes \CH_i(X)$, $1\leq i\leq r-1$ and $\Gamma'''$ is supported on $Z\times X$ for some closed subset $Z\subset X$ of codimension $r$. Then we carry out the symmetrization argument $\Delta_X = \Delta_X\circ {}^t\Delta_X$. We note that $\Gamma'_i\circ (-)$ (resp. $(-)\circ {}^t\Gamma'_i$) is again decomposable and contained in $\CH^i(X)\otimes\CH_i(X)$ ({resp.} in $\CH_i(X)\otimes \CH^i(X)$). Then we carry out a similar argument as before to show that $\Gamma'''\circ{}^t\Gamma'''$ is of the required form. First, there exist smooth projective varieties $D'_i$ of dimension $d-r$, correspondences $\Gamma'''_i\in \CH_d(D'_i\times X)$ and morphism $f_i:D'_i\rightarrow X$ such that
\[
\Gamma''' = \sum_i (f_i,Id_X)_*\Gamma'''_i = \sum_i\Gamma'''_i\circ f_i^*.
\]
Then we have
\[
 \Gamma'''\circ{}^t\Gamma''' = \sum_{i,j} \Gamma'''_i\circ f_i^*\circ f_{j,*}\circ {}^t\Gamma'''_j
\]
We write down the cycles $f_i^*\circ f_{j,*}\in \CH_{d-2r}(D'_i\times D'_j)$ explicitly. Then the terms of the above sum can be grouped into the following types. 

\textit{Term type 1}: $\Gamma_1\circ {}^t\Gamma_2 + \Gamma_2\circ {}^t\Gamma_1$, where $\Gamma_1,\Gamma_2\in \CH_{d-r}(Z\times X)$ for some smooth projective variety $Z$ of dimension $d-2r$. Such terms appear in $\Gamma'''_i\circ f_i^*\circ f_{j,*}\circ {}^t\Gamma'''_j + \Gamma'''_j\circ f_j^*\circ f_{i,*}\circ {}^t\Gamma'''_i$ where $i<j$. Such a type 1 term can be written as
\[
(\Gamma_1 +\Gamma_2)\circ {}^t(\Gamma_1 + \Gamma_2) -\Gamma_1\circ {}^t\Gamma_1 -\Gamma_2\circ {}^t\Gamma_2.
\]
Thus such a term can be written as the required form. 

We still need to deal with the terms $\Gamma'''_i\circ f_i^*\circ f_{i,*}\circ {}^t\Gamma'''_i$. By Lemma \ref{lem symmetric}, such a term can also be written as $\Gamma'''_i\circ \Sigma_i\circ {}^t\Gamma'''_i$ where $\Sigma_i$ is a symmetric cycle of dimension $d-2r$ on $D'_i\times D'_i$. Then $\Sigma_i$ is a linear combination of cycles of the following form: (a) $Z+{}^tZ$ where $Z\subset D'_i\times D'_i$ is an irreducible subvariety of dimension $d-2r$; (b) A subvariety $Z\subset D'_i\times D'_i$ of dimension $d-2r$, which is contained in the diagonal $\Delta_{D'_i}$; (c) An irreducible subvariety $Z\subset D'_i\times D'_i$, which is not supported on the diagonal but invariant under the involution of switching the two factors of $D'_i\times D'_i$. After resolution of the singularities of $Z$ if necessary, we see the following: (a) produces terms of type 1;  (b) produces terms for the form $\Gamma\circ {}^t\Gamma$, which satisfies the required form; (c) produces a new term of type 2 as follows.

\textit{Term type 2:} $\Gamma\circ \sigma \circ {}^t\Gamma$, where $\Gamma\in \CH_{d-r}(Z\times X)$ for some smooth projective variety $Z$ of dimension $d-2r$ and $\sigma$ is a nontrivial involution of $Z$ (induced by the involution of $D'_i\times D'_i$).
\end{proof}

\begin{rmk}
In Proposition \ref{prop voisin}, only diagonalized terms $\Gamma_i\circ {}^t\Gamma_i$ appear since we are allowed to blow up $X$ to make the $D_i$'s to be normal crossing. However, blow-up only preserves universal triviality of  $\CH_0$ and hence is not allowed in Proposition \ref{prop voisin generalized}. Thus the image of $D_i$ in $X$ can fail to be normal and that produces the terms $\Gamma_i\circ\sigma_i\circ {}^t\Gamma_i$ that are not diagonalized.
\end{rmk}

\begin{cor}\label{cor voisin generalized}
Let $X$ be a smooth projective variety of dimension $d$ over $\C$ such that $\CH_i(X)$ is universally trivial for all $i=0,1,\ldots,r-1$. Then the following are true.
\begin{enumerate}[(i)]
\item There exist smooth projective varieties $Z_i$ of dimension $d-2r$, cycles $\Gamma_i\in \CH_{d-r}(Z_i\times X)$ and integers $n_i\in \Z$ such that $\Gamma =\sum n_i \Gamma_i$ induces a universally surjective homomorphism $\bigoplus_i\CH_0(Z_i)\rightarrow \CH_r(X)$ and
\[
 \sum n_i \langle\Gamma_i^*\alpha,\sigma_i^*\Gamma_i^*\beta \rangle_{Z_i} = \langle \alpha, \beta \rangle_X
\]
for all $\alpha,\beta\in \HH^d(X,\Z)$, where $\sigma_i:Z_i\rightarrow Z_i$ is either the identity map or an involution. In particular,
\[
F^1\HH^d(X,\Z) = F^2\HH^d(X,\Z) = \cdots =\F^{2r}\HH^d(X,\Z) = 0.
\]
\item $\HH^{p,q}(X)=0$ for all $p\neq q$ and $\min\{p,q\}< r$.
\item $F^i\HH^d(X,\Z)=0$, for all $i=0,1,\ldots, 2r$. 
\end{enumerate}
\end{cor}

\begin{proof}
Let $L/k$ be an arbitrary field extension. Let $\gamma\in\CH_r(X_L)$, then
\[
\gamma = (\Delta_X)_*\gamma = \Sigma_*\gamma + \Gamma_*\gamma = \Gamma_*\gamma.
\]
Here all the correspondences are understood to be their base changes to $L$; we also use the fact $\Sigma_*\gamma= 0$ which is a consequence of 
$$
\Sigma\in \bigoplus_{i<r}\left( \CH^i(X)\otimes\CH_i(X) \oplus \CH_i(X)\otimes\CH^i(X) \right)\qquad \text{and} \qquad \gamma\in \CH_r(X_L).
$$
Since $\Gamma_*\gamma$ factors through $\sum \CH_0((Z_i)_L)$ the universal surjection in statement (i) follows immediately. Let $\alpha,\beta\in \HH^d(X,\Z)$. Then we have
\[
\langle\alpha,\beta\rangle_X = [\Delta_X]\cup (\alpha\otimes\beta) = [X\times x + x\times X +\Sigma +\Gamma]\cup(\alpha\otimes\beta) = [\Gamma]\cup (\alpha\otimes\beta).
\]
Note that here $\Sigma\cup(\alpha\otimes\beta) =0$ since the algebraic cycle $\Sigma$ is of the special form as above. Indeed, if $\Sigma=\Sigma_1\otimes\Sigma_2\in \CH^i(X)\otimes\CH_i(X)$, then $\Sigma\cup (\alpha\otimes\beta) = \langle[\Sigma_1],\alpha\rangle_X \langle[\Sigma_2],\beta\rangle_X =0$ since $2i<d$. At the same time, we have
\[
[\Gamma]\cup (\alpha \otimes\beta) = \left[\sum n_i\Gamma_i \circ \sigma_i \circ {}^t\Gamma_i \right] \cup (\alpha\otimes\beta) = \sum n_i \langle\Gamma_i^*\alpha,\sigma_i^*\Gamma_i^*\beta \rangle_{Z_i}.
\]
Let $\omega\in\HH^{p,q}(X)$. Then we have
\[
\omega = (\Delta_X)_*\omega = \left( x\times X + X\times x + \Sigma \right)_*\omega + \Gamma_*\omega,
\]
where the first term vanishes whenever $p\neq q$. If $\max \{p,q\}<r$, then the second term also vanishes since it factors through $\HH^{p-r,q-r}(Z_i)$. Statement (ii) follows. Let $\alpha\in F^i\HH^d(X,\Z)$ where $i\leq 2r$. Since $\Delta_X$ factors through varieties of dimension at most $d-2r$, it follows that $\alpha = (\Delta_X)^* \alpha = 0$. Thus (iii) is proved.
\end{proof}

\begin{rmk}
The statement ({ii}) holds under the weaker assumption that $\CH_i(X)_{\Q}$ is universally trivial for $i=0,1,\ldots,r-1$.
\end{rmk}

\begin{thm}\label{thm theta}
Let $X$ be a smooth projective variety of dimension $d$ over $\C$ such that $\CH_i(X)$ is universally trivial for all $i=0,1,\ldots,r-1$, where $r$ is a positive integer with $2r\leq d$. Let $F$ be a smooth projective variety and let $\gamma\in \CH_r(X_{\C(F)})$ be a universally generating $r$-cycle. Then there exists a symmetric algebraic cycle $\theta\in \CH_{d-2r}(F\times F)$ such that
\[
 [\theta]\cdot (\hat\alpha \otimes \hat\beta) = \langle \alpha,\beta \rangle_X
\]
for all $\alpha,\beta\in F^{2r+1}\HH^d(X,\Z)$, where $\hat\alpha := \Gamma^* \alpha$ for some spreading $\Gamma\in\CH^r(F\times X)$ of $\gamma$.
\end{thm}

\begin{proof}
By Corollary \ref{cor voisin generalized}, there exist smooth projective varieties $Z_i$ of dimension $d-2r$, cycles $\Gamma_i\in \CH_{d-r}(Z_i\times X)$ and integers $n_i$ such that
\[
\sum n_i \langle\Gamma_i^*\alpha, \sigma_i^*\Gamma_i^*\beta \rangle_{Z_i} = \langle \alpha, \beta \rangle_X
\]
for all $\alpha,\beta\in \HH^d(X,\Z)$. Let $\gamma_i:=\Gamma_i|_{\eta_{Z_i}} \in \CH_r(X_{K_i})$, where $K_i=\C(Z_i)$. For each $i$, by the universal generating property of $\Gamma$, there exists $\tau_i\in \CH_0(F_{K_i})$ such that $\Gamma_*\tau_i = \gamma_i$. Let $T_i\in \CH_{d-2r}(Z_i\times F)$ be a spreading of $\tau_i$ and set
\[
 \theta = \sum n_i T_i\circ\sigma_i\circ {}^t T_i \in \CH_{d-2r}(F \times F).
\]
Note that the condition $\Gamma_*\tau_i = \gamma_i$ means that $\Gamma'_i:=\Gamma\circ T_i$ and $\Gamma_i$ agrees over $\eta_{Z_i}$. Hence $\Gamma_i-\Gamma'_i$ factors through a divisor of $Z_i$. It follows that $\Gamma_i^*\alpha = {\Gamma'}_i^*\alpha$ for all $\alpha\in F^{2r+1}\HH^d(X,\Z)$. Thus, for $\alpha,\beta\in F^{2r+1}H^d(X,\Z)$, we have
\begin{align*}
 [\theta] \cdot (\hat\alpha \otimes \hat\beta) & = \sum n_i [T_i\circ\sigma_i\circ {}^t T_i] \cup (\Gamma^*\alpha \otimes \Gamma^*\beta)\\
  & = \sum n_i \langle T_i^*\Gamma^*\alpha, \sigma_i^* T_i^*\Gamma^*\beta \rangle_{Z_i}\\
  & = \sum n_i \langle {\Gamma'}_i^*\alpha, \sigma_i^* {\Gamma'}_i^{*}\beta \rangle_{Z_i}\\
  & = \sum n_i \langle {\Gamma}_i^*\alpha, \sigma_i^* {\Gamma}_i^{*}\beta \rangle_{Z_i}\\
  & = \langle \alpha,\beta \rangle_X.
\end{align*}
This finishes the proof.
\end{proof}

\section{Universal generation of 1-cycles on cubic hypersurfaces}

In this section we show that the Chow group of 1-cycles on a smooth cubic hypersurface is universally generated by lines. Let $K$ be an arbitrary base field of any characteristic. Let $X\subseteq \PP^{d+1}_{K}$ be a smooth cubic hypersurface of dimension $d$ over $K$. Let $F=F(X)$ be the variety of lines on $X$. It is known by \cite{ak} that $F/K$ is smooth of dimension $2d-4$. Let
\[
\xymatrix{
P\ar[r]^q\ar[d]_p &X\\
F &
}
\]
be the universal line.

\begin{thm}\label{thm cubic generation}
Let $K$ be an arbitrary field and $X/K$ a smooth cubic hypersurface of dimension $d\geq 3$. Then the following are true. 
\begin{enumerate}[(i)]
\item If $\gamma\in \CH_1(X)$, then $3\gamma \in q_*p^* \CH_0(F) + \Z\, h^{d-1}$.
\item If $\CH_1(X)$ contains an element of degree not divisible by 3, then 
\[
\CH_1(X) = q_*p^*\CH_0(F) + \Z\, h^{d-1}.
\]
\item If $\CH_0(F)$ contains an element of degree 1, then
\[
P_*=q_*p^*: \CH_0(F)\longrightarrow \CH_1(X)
\]
is universally surjective.
\end{enumerate}
\end{thm}

The key ingredient one needs to prove the above universal generation is the following relations among 1-cycles on $X$.

\begin{prop}\label{prop relation}
Let $\gamma_1,\gamma_2\in \CH_1(X)$ be 1-cycles of degree $e_1$ and $e_2$ respectively. Let $h\in \CH^1(X)$ be the class of a hyperplane section.
\begin{enumerate}[(i)]
\item There exists a 0-cycle $\mathfrak{a}\in \CH_0(F)$ such that
\[
 (2e_1-3)\gamma_1 + q_*p^*\mathfrak{a} = a h^{d-1}, \quad\text{in }\CH_1(X),
\]
for some integer $a$. If $\gamma_1$ is represented by a geometrically irreducible curve $C$ in general position, then $\mathfrak{a}$ can be taken to be all the lines that meet $C$ in two points.
\item We have
\[
 2e_2\gamma_1 + 2e_1\gamma_2 + q_*p^*\mathfrak{a}' = 3e_1e_2\, h^{d-1}, \quad\text{in } \CH_1(X),
\]
where $\mathfrak{a}' = p_*q^*\gamma_1\cdot p_*q^*\gamma_2\in \CH_0(F)$ is a 0-cycle of degree $5e_1e_2$.
\item Let $\xi\in \CH^r(X)$ with $r<d-1$. Then
\[
 2e_1\xi + q_*p^*\mathfrak{a}'' = b h^{r}, \quad\text{in } \CH^r(X),
\]
where $\mathfrak{a}'' = p_*q^*\gamma_1\cdot p_*q^*\xi \in \CH^{d+r-3}(F)$ and $b\in \Z$.
\end{enumerate}
\end{prop}

We now prove Theorem \ref{thm cubic generation} assuming Proposition \ref{prop relation}.

\begin{proof}[Proof of Theorem \ref{thm cubic generation}]
Apply Proposition \ref{prop relation} (i) to $\gamma_1=\gamma$, we get $(2e-3)\gamma \in q_*p^*\CH_0(F) + \Z h^{d-1}$. Apply Proposition \ref{prop relation} (ii) to the case of $\gamma_1=\gamma$ and $\gamma_2=h^{d-1}$, we get $6\gamma \in q_*p^*\CH_0(F) + \Z h^{d-1}$. Note that the greatest common divisor of $6$ and $2e-3$ is either 3 or 1. We see that $3\gamma \in q_*p^*\CH_0(F) + \Z h^{d-1}$. This proves (i).

If $e=\deg(\gamma)$ is not divisible by $3$, then the greatest common divisor of 6 and $2e-3$ is 1. Hence $\gamma \in q_*p^*\CH_0(F) + \Z h^{d-1}$. Now let $\gamma'\in\CH_1(X)$ be an arbitrary 1-cycle. We apply Proposition \ref{prop relation} (ii) to $\gamma$ and $\gamma'$ and get
\[
2e \gamma' + 2e'\gamma \in q_*p^*\CH_0(F) + \Z h^{d-1}.
 \] 
Since $\gamma \in q_*p^*\CH_0(F) + \Z h^{d-1}$, we conclude that $2e \gamma' \in q_*p^*\CH_0(F) + \Z h^{d-1}$. We already know that $3\gamma' \in q_*p^*\CH_0(F) + \Z h^{d-1}$. By assumption, the greatest common divisor of 3 and $2e$ is 1. We conclude $\gamma'\in q_*p^*\CH_0(F) + \Z h^{d-1}$, which establishes (ii).

Since our base field $K$ is arbitrary, for (iii) it suffices to show that $q_*p^*:\CH_0(F)\rightarrow \CH_1(X)$ is surjective. Let $\mathfrak{a}_0\in \CH_0(F)$ be an element of degree 1. Take $\mathfrak{l}=q_*p^*\mathfrak{a}_0\in \CH_1(X)$, which is of degree 1. By (ii), we only need to show that $h^{d-1}$ is contained in $q_*p^*\CH_0(F)$. Apply Proposition \ref{prop relation} (ii) to $\gamma_1=\gamma_2=h^{d-1}$, we get
\[
 6h^{d-1} + 6h^{d-1} + q_*p^*\mathfrak{a}_1 = 27 h^{d-1},\quad \mathfrak{a}_1\in \CH_0(F).
\]
It follows that $15 h^{d-1} = q_*p^*\mathfrak{a}_1$. We apply the proposition again to $h^{d-1}$ and $\mathfrak{l}$ and get
\[
 2h^{d-1} + 6\mathfrak{l} + q_*p^*\mathfrak{a}_2 = 9 h^{d-1},\quad \mathfrak{a}_2\in \CH_0(F).
\]
Thus $7h^{d-1}=q_*p^*(\mathfrak{a}_2 + 6\mathfrak{a}_0)$. Thus we conclude that $h^{d-1}\in q_*p^*\CH_0(F)$. This proves (iii).
\end{proof}

The proof of Proposition \ref{prop relation} over an algebraically closed field is given in \cite{relation, pt}. In this section, we develop a universal version of the techniques in \cite{relation} which leads to a proof of Proposition \ref{prop relation}. We would like to study the geometry of $X^{[2]}$, the Hilbert scheme of two points on $X$. Let $\delta\in\CH^1(X^{[2]})$ be the ``half boundary". Namely, $2\delta$ is the class of the boundary divisor parametrizing nonreduced length-$2$ subschemes. 

Following Galkin--Shinder \cite{gs}, we define a rational map
\[
\Phi : X^{[2]} \dashrightarrow P_{X}:=\PP(\mathcal{T}_{\PP^{d+1}_K}|_{X})
\]
as follows. Let $x,y\in X$ be two distinct points and they determine a line $L_{x,y} \subset \PP^{d+1}_{K}$. If $L_{x,y}$ is not contained in $X$, then $L_{x,y}$ intersects $X$ in a third point $z\in X$. Then $\Phi([x,y])$ is the point represented by the 1-dimensional subspace $\mathcal{T}_{L_{x,y},z}$ in $\PP(\mathcal{T}_{\PP^{d+1},z})$. Let $P^{[2]}_{F}\subset X^{[2]}$ be the relative Hilbert scheme of $P/F$. Let $p': P^{[2]}_{F}\rightarrow F$ be the structure morphism. Then it is a fact that the indeterminacy of $\Phi$ can be resolved by blowing up $X^{[2]}$ along $P^{[2]}_{F}$. This is proved in Voisin \cite[Proposition 2.9]{voisin universal} for the case $K=\C$ and the same proof works over an arbitrary field $K$. The resulting morphism $\tilde{\Phi}: \widetilde{X^{[2]}}\rightarrow P_{X}$ is the blow up of $P_{X}$ along $P\subset P_{X}$. The inclusion $i_1:P\hookrightarrow P_{X}$ sends a point $(x\in l)$ to the direction $\mathcal{T}_{l,x}$ in $\mathcal{T}_{X,x}$. The two blow-ups $\tau: \widetilde{X^{[2]}} \rightarrow X^{[2]}$ and $\tilde{\Phi} : \widetilde{X^{[2]}} \rightarrow P_X$ share the same exceptional divisor $E \subset \widetilde{X^{[2]}}$. To summarize, we have the following commutative diagram
\[
\xymatrix{
 E\ar[rr]^{\pi_1}\ar[dd]_{\pi_2}\ar[dr]^j &&P\ar[d]_{i_1}\\
 &\widetilde{X^{[2]}}\ar[r]^{\tilde{\Phi}}\ar[d]_{\tau} &P_{X}\\
 P^{[2]}_{F}\ar[r]^{i_2} &X^{[2]}\ar@{-->}[ru]_{\Phi} &
}
\]
We also have the natural identification
\[
 E = P\times_{F} P^{[2]}_{F}
\]
and the morphisms $\pi_1$ and $\pi_2$ are the two projections.

There is a double cover $\sigma: \widetilde{X\times {X}} \rightarrow X^{[2]}$, where $\widetilde{X\times {X}}$ is the blow up of $X\times {X}$ along the diagonal. There is also a morphism 
\[
\Psi:\widetilde{X\times {X}}\longrightarrow P_{X}, \qquad (x,y)\mapsto [\mathcal{T}_{L_{x,y},x}].
\]
Note that the composition $\widetilde{X\times {X}} \rightarrow P_{X} \rightarrow X$, is the blow-up morphism $\rho$ followed by the projection onto the first factor. The above morphisms form the following commutative diagram.
\[
\xymatrix{
 X^{[2]} &\widetilde{X \times X}\ar[l]_{\sigma} \ar[r]^{\quad\Psi} \ar[d]_{\rho} &P_{X}\ar[d]^{\pi} \\
  &X \times X \ar[r]^{\quad p_1} &X
}
\]
Given algebraic cycles $\alpha,\beta\in \CH_*(X)$, we write
\[
 \alpha \hat\otimes \beta : = \sigma_*\rho^*(\alpha \times \beta)\in \CH_*(X^{[2]}).
\]

Two distinct points $x,y\in {X}$ determine a line $L_{x,y}$ in $\PP^{d+1}$. This defines a morphism
\[
\varphi: X^{[2]}\longrightarrow G(2, {d+2}),
\]
where $G(2,{d+2})$ is the Grassmannian of rank two subspaces of $K^{d+2}$. Together with the previous morphisms, we have a commutative diagram with all squares being fiber products.
\begin{equation}\label{eq diagram 1}
\xymatrix{
 \widetilde{X\times {X}} \cup \widetilde{X^{[2]}}\ar[rr]\ar[d] &&P_{X}\ar[r]^{\pi}\ar[d]^{i'} &X\ar[d]^{i}\\
 Q\ar[rr]^{\varphi'}\ar[d]^{f} &&G(1,2, {d+2})\ar[r]^{\quad \tilde{\pi}}\ar[d]^{\tilde{f}} &\PP^{d+1}_K\\
 X^{[2]}\ar[rr]^{\varphi} &&G(2, {d+2}) &
}
\end{equation}
Let $h\in \CH^1(\PP^{d+1}_K)$ be the class of a hyperplane section. We still use $h\in \CH^1(X)$ to denote the restriction of $h$ to $X$ and let $h_{Q}\in \CH^{1}(Q)$ be the pull-back of $h$ to $Q$ via the natural morphism $Q\rightarrow \PP^{d+1}_K$. 

\begin{lem}\label{lem divisorial relation}
We have
\[
E = - h_{Q}|_{\widetilde{X^{[2]}}} - \tau^*(2h\hat\otimes 1 -3\delta)
\]
in $\CH^1(\widetilde{X^{[2]}})$.
\end{lem}

\begin{proof}
Let $\mathscr{E}\subset \calO_G^{d+2}$ be the tautological rank-2 subbundle on $G:=G(2,{d+2})$. Note that $Q\cong\PP(\mathscr{E}|_{X^{[2]}})$ is a $\PP^1$-bundle over $X^{[2]}$. The projective bundle formula for Chow groups implies that, as divisors on $Q$, 
\[
 \widetilde{X^{[2]}} = h_{Q} + f^*\mathfrak{a}, \quad \text{in }\CH^1(Q),
\]
for some $\mathfrak{a}\in \CH^1(X^{[2]})$.  We also have the following short exact sequence (derived from the Euler sequence)
\[
 \xymatrix{
  0\ar[r] &\calO(h_{Q})\otimes \frac{f^*(\mathscr{E}|_{X^{[2]}})}{\calO(-h_{Q})} \ar[r] &\mathcal{T}_{Q}\ar[r] & f^*\mathcal{T}_{X^{[2]}}\ar[r] &0.
 }
\]
As a consequence, we have
\begin{align*}
 K_{Q} &= -c_1(\mathcal{T}_{Q}) \\
  & = f^*K_{X^{[2]}} -2h_{Q} - f^*c_1(\mathscr{E}|_{X^{[2]}}).
\end{align*}
The canonical divisor of $\widetilde{X^{[2]}}$ can be computed by the adjunction formula as
\begin{align*}
 K_{\widetilde{X^{[2]}}} & = (K_{Q} + \widetilde{X^{[2]}})|_{\widetilde{X^{[2]}}}\\
 & = \tau^*K_{X^{[2]}} - h_{Q}|_{\widetilde{X^{[2]}}} -\tau^*c_1(\mathscr{E}|_{X^{[2]}}) + \tau^*\mathfrak{a}.
\end{align*}
Since $\widetilde{X^{[2]}}$ is the blow-up of $X^{[2]}$, we also have
\[
 K_{\widetilde{X^{[2]}}} = \tau^*K_{X^{[2]}} +E.
\]
Comparing the above two expressions, we get
\begin{equation}\label{eq Ea}
E = -h_{Q}|_{\widetilde{X^{[2]}}} - \tau^*c_1(\mathscr{E}|_{X^{[2]}}) + \tau^*\mathfrak{a},  \quad \text{in } \CH^1(\widetilde{X^{[2]}}).
\end{equation}
To determine the value of $\mathfrak{a}$, we apply $\tau_*$ to the above equation and get
\begin{equation}\label{eq tau push}
 0 = -f_*(h_{Q}\cdot \widetilde{X^{[2]}}) - c_1(\mathscr{E}|_{X^{[2]}}) + \mathfrak{a}
\end{equation}
in $\CH^1(X^{[2]})$. The first term on the right hand side can be computed as follows.
\begin{align*}
 f_*(h_{Q}\cdot \widetilde{X^{[2]}}) & = f_*\left(h_{Q}\cdot(\widetilde{X^{[2]}} + \widetilde{X \times X}) \right) - f_*\left(h_{Q}\cdot\widetilde{X \times X} \right)\\
 & = f_*(h_{Q} \cdot 3h_{Q}) -\sigma_*\rho^*(h\otimes 1)\\
 & = 3f_*h^2_{Q} - h\hat{\otimes} 1.
\end{align*}
Note that $Q = \PP(\mathscr{E}|_{X^{[2]}})$ is a $\PP^1$-bundle over $X^{[2]}$ and $h_{Q}$ is the class of the associated relative $\calO(1)$-bundle. Thus we have the equation
\[
 h^2_Q + f^*c_1(\mathscr{E}|_{X^{[2]}})\cdot h_{Q} + f^*c_2(\mathscr{E}|_{X^{[2]}}) = 0
\]
and it follows that
\[
 f_*h^2_{Q} = -c_1(\mathscr{E}|_{X^{[2]}}).
\]
This combined with equation \eqref{eq tau push}, we get
\[
\mathfrak{a} = -2c_1(\mathscr{E}|_{X^{[2]}}) - h\hat\otimes 1.
\]
We plug this into \eqref{eq Ea} and get
\begin{equation}\label{eq div rel 1}
E = - h_{Q}|_{\widetilde{X^{[2]}}} - \tau^*(h\hat\otimes 1) - 3 \tau^* c_1(\mathscr{E}|_{X^{[2]}}),\quad \text{in }\CH^1(\widetilde{X^{[2]}}).
\end{equation}

We still need to compute $c_1(\mathscr{E}|_{X^{[2]}})$. For simlicity, let $V\cong K^{d+2}$ be a $(d+2)$-dimensional vector space with a fixed identification $\PP^{d+1}_K = \PP(V)$. For any coherent sheaf $\mathscr{F}$ on $X$, we define
\[
 \mathscr{F}^{[2]} : = \sigma_*\rho^*p_1^*\mathscr{F}.
\]
If $\mathscr{F}$ is locally free of rank $r$, then $\mathscr{F}^{[2]}$ is locally free of rank $2r$. The inclusion $\calO_{X}(-1) \hookrightarrow {V}\otimes_K \calO_X$ induces an inclusion
\begin{equation}\label{eq inclusion 1}
\calO(-1)^{[2]} \hookrightarrow {V}\otimes \calO^{[2]}_{X}= {V}\otimes \sigma_*\calO_{\widetilde{X\times X}}. 
\end{equation}
Recall that $\delta\in\CH^1(X^{[2]})$ is the ``half boundary" which fits into the following short exact sequence
\begin{equation}\label{eq delta}
 \xymatrix{
  0\ar[r] &\calO_{X^{[2]}} \ar[r] 
  &\sigma_*\calO_{\widetilde{X\times X}}\ar[r] &\calO_{X^{[2]}}(-\delta)\ar[r] &0. 
  }
\end{equation}
Applying $V\otimes_K -$ to the above sequence, we get
\begin{equation}\label{eq V delta}
\xymatrix{
0\ar[r] &V\otimes \calO_{X^{[2]}} \ar[r] 
  &V\otimes \sigma_*\calO_{\widetilde{X\times X}}\ar[r] &V\otimes \calO_{X^{[2]}}(-\delta)\ar[r] &0
  }
 \end{equation} 
 Pulling it back to $\widetilde{X\times X}$, we get
\begin{equation}\label{eq pull V delta}
\xymatrix{
 0\ar[r] &V\otimes\calO_{\widetilde{X\times X}} \ar[r] &V\otimes \sigma^*\sigma_*\calO_{\widetilde{X\times X}} \ar[r] & V\otimes \calO_{\widetilde{X\times X}}(-D)\ar[r] &0,
}
\end{equation}
where $D\subset \widetilde{X\times X}$ is the exceptional divisor of the diagonal blow-up $\rho: \widetilde{X\times X} \rightarrow X\times X$. The short exact sequence \eqref{eq pull V delta} restricted to $\widetilde{X\times X} \backslash D$ becomes
\[
\xymatrix{
0\ar[r] &V\ar[r] &V\oplus V \ar[r] &V\ar[r] &0,
}
\]
where the first map is $v\mapsto (v,v)$ and the second map is $(v,v')\mapsto v-v'$. Consider the restriction of the second map of \eqref{eq pull V delta} to the canonical rank two subbundle $\rho^*(p_1^*\calO_X(-1) \oplus p_2^*\calO_X(-1)) \hookrightarrow (V\oplus V)\otimes \calO_{\widetilde{X\times X}}$ and we get
\[
\mu: \rho^*(p_1^*\calO_X(-1) \oplus p_2^*\calO_X(-1)) \longrightarrow V\otimes \calO_{\widetilde{X\times X}}(-D)
\]
Let $x$ and $y$ be two distinct points of $X$. Then the image of $\mu$ at the point $(x,y)$ is simply $V_x+V_y \subset V$, where $V_x\subset V$ (\textit{resp}. $V_y\subset V$) is the one dimensional subspace associated to $x\in \PP(V)$ (\textit{resp.} $y$). Thus $V_x+V_y\subset V$ is the two dimensional subspace associated to the line $L_{x,y}$ passing through $x$ and $y$. Thus $\varphi^*\mathscr{E}$ should be the unique rank two subbundle of $V\otimes \calO_{X^{[2]}}$ which pulls back to $\rho^*(p_1^*\calO_X(-1) \oplus p_2^*\calO_X(-1))$ on $\widetilde{X\times X}\backslash D$.

To find $\varphi^*\mathscr{E}$, we combine \eqref{eq V delta} with the inclusion \eqref{eq inclusion 1} and get a rank two subbundle
\[
\calO(-1)^{[2]}\otimes \calO(\delta) \rightarrow {V}\otimes \calO_{X^{[2]}}.
\]
Note that 
\[
\sigma^*\calO(-1)^{[2]} = \sigma^*\sigma_*\rho^*p_1^*\calO_X(-1) = \rho^*(p_1^*\calO_X(-1) \oplus p_2^*\calO_X(-1)).
\]
Thus the rank 2 subbundle of $V\otimes \calO_{X^{[2]}}$ constructed above satisfies the required condition and hence
\[
 \varphi^*\mathscr{E} \cong \calO(-1)^{[2]}\otimes \calO(\delta).
\]
The Grothendieck--Riemann--Roch formula gives
\[
c_1(\calO(-1)^{[2]}) = -h\hat\otimes 1 - \delta.
\]
The Grothendieck--Riemann--Roch formula uses rational coefficients, but this is fine here since $\CH^1(X^{[2]})$ is torsion free. Another way to get the above formula is to apply the exact functor $(-)^{[2]}$ to the short exact sequence $\xymatrix{0\ar[r] & \calO(-1)\ar[r] &\calO\ar[r] &\calO_H\ar[r] &0}$ and then take first Chern classes. It follows from the above equation that
\[
c_1(\mathscr{E}|_{X^{[2]}}) = c_1(\varphi^*\mathscr{E}) = -h\hat\otimes 1 +\delta
\]
We combine this with equation \eqref{eq div rel 1} and prove the lemma.
\end{proof}

\begin{lem}\label{lem Gamma}
Let $\Gamma\in\CH^r(X^{[2]})$. 
\begin{enumerate}[(i)]
\item The following equation holds in $\CH^{r+1}(P_{X})$,
\[
\tilde{\Phi}_*(E\cdot\tau^*\Gamma) = (i_1)_*p^*\gamma,
\]
where $\gamma = p'_*(i_2)^*\Gamma$ in $\CH^{r-2}(F)$. In particular, we have
\[
 \pi_*\tilde{\Phi}_*(E\cdot\tau^*\Gamma) = q_*p^*\gamma, \quad\text{in }\CH^{r-3}(X).
\]

\item We have
\[
 \tilde{\Phi}_*\tau^*\Gamma = (i')^*\tilde{f}^*\varphi_*\Gamma - \Psi_*\sigma^*\Gamma,\quad \text{in } \CH^r(P_{X}).
\]
\end{enumerate}
\end{lem}
\begin{proof}
The statement (\textit{i}) follows from a straightforward calculation as follows.
\begin{align*}
\tilde{\Phi}_*(E\cdot\tau^*\Gamma) & = \tilde{\Phi}_*j_*j^*\tau^*\Gamma\\
 & = (i_1)_*(\pi_1)_* \pi_2^*i_2^*\Gamma\\
 & = (i_1)_*p^*p'_*i_2^*\Gamma.
\end{align*}
We also have
\begin{align*}
 \tilde{\Phi}_* \tau^*\Gamma & = \left(\tilde{\Phi}_* \tau^*\Gamma  + \Psi_*\sigma^*\Gamma\right) - \Psi_*\sigma^*\Gamma\\
 & = (i')^*\tilde{f}^*\varphi_*\Gamma - \Psi_*\sigma^*\Gamma.
\end{align*}
This proves (\textit{ii}).
\end{proof}

\begin{lem}\label{lem existence Gamma}
Let $\{\tilde{\Gamma}_1,\ldots,\tilde{\Gamma}_n, \tilde{\Sigma}_1,\ldots,\tilde{\Sigma}_m\}$ be a set of distinct irreducible reduced curves in $X$. Let $\tilde{\Gamma}=\sum n_k \tilde{\Gamma}_k$ and $\tilde{\Sigma}=\sum m_k \tilde{\Sigma}_k$ be two algebraic cycles of dimension one in $X$. Let $e$ (resp. $e'$) be the degree of $\tilde{\Gamma}$ (resp. $\tilde{\Sigma}$) as a $1$-cycle on $X$.
\begin{enumerate}[(i)]
\item There exists an algebraic cycle $\Gamma$ on $X^{[2]}$ such that 
\begin{align*}
\pi_*\Psi_*\sigma^*\Gamma & = 0,\\ 
\pi_*\Psi_*\sigma^*(h\hat\otimes 1 \cdot \Gamma) & = e\, \tilde{\Gamma},\\
\pi_*\Psi_*\sigma^*(\delta \cdot \Gamma) & = \sum n_k^2 \tilde{\Gamma}_k.
\end{align*}
\item There exists an algebraic cycle $\Gamma'$ on $X^{[2]}$ such that 
\begin{align*}
\pi_*\Psi_*\sigma^*\Gamma' & = 0,\\
\pi_*\Psi_*\sigma^*(h\hat\otimes 1 \cdot \Gamma') & = e'\, \tilde{\Gamma} + e\, \tilde{\Sigma},\\
\pi_*\Psi_*\sigma^*(\delta \cdot \Gamma') & = 0.
\end{align*}
\item Let $\tilde{\Xi}$ be an algebraic cycle of codimension $r<d-1$ on $X$. Then there exists an algebraic cycle $\Gamma''$ on $X^{[2]}$ such that
\begin{align*}
\pi_*\Psi_*\sigma^*\Gamma'' & = 0\\
 \pi_*\Psi_*\sigma^*(h\hat\otimes 1 \cdot \Gamma'') & = e\, \tilde{\Xi},\\
\pi_*\Psi_*\sigma^*(\delta \cdot \Gamma'') & = 0.
\end{align*}
\end{enumerate} 
\end{lem}
\begin{proof}
Let $\tilde{\Gamma}_{kl}: = \tilde{\Gamma}_k\times \tilde{\Gamma}_l \subset X\times X$. Let $\Gamma_{kl}\subset \widetilde{X\times {X}}$ be the strict transform of $\tilde{\Gamma}_{kl}$. Note that $\Gamma_{kl} + \Gamma_{lk} = \sigma^*\sigma_*\Gamma_{kl}$ for all $k<l$. The cycle $\Gamma_{ll}$ maps with degree two onto its image. By abuse of notation we use $\Gamma_l^{[2]}\subset X^{[2]}$ to denote this image. Then we have $\Gamma_{ll} = \sigma^* \Gamma_l^{[2]}$. Then the algebraic cycle $\Gamma=\sum_{k<l} n_kn_l \sigma_*\Gamma_{kl} + \sum n_l^2\Gamma_l^{[2]}$ on $X^{[2]}$ satisfies $\sigma^*\Gamma =\sum n_kn_l \Gamma_{kl}$.  Thus
\[
\pi_*\Psi_*\sigma^*\Gamma =\sum_{k,l}n_kn_l (p_{1})_*\rho_*\Gamma_{kl}  = 0
\]
and
\[
\pi_*\Psi_*\sigma^*(\Gamma\cdot h\hat\otimes 1) = \sum n_kn_l (p_1)_*\rho_*\Big(\Gamma_{kl}\cdot \rho^*(h\otimes 1 + 1\otimes h)\Big) = \sum n_k n_l e_k\tilde{\Gamma}_l = e\,\tilde\Gamma,
\]
where $e_k = \deg(\tilde{\Gamma}_k|_{X_{\eta_B}})$, $e=\sum n_ke_k$.
The remaining equation holds because 
\[
 \pi_*\Psi_*(\Gamma_{kl}\cdot \sigma^*\delta) = \begin{cases}
 0, &\text{if } k\neq l;\\
 \tilde{\Gamma}_k, &\text{if } k=l.
 \end{cases}
\]
This proves (i). Statements (ii) and (iii) are proved similarly. For example, $\Gamma'$ is obtained as the image of the strict transform of $\tilde{\Gamma} \times \tilde{\Sigma}$ and $\Gamma''$ is obtained as the image of the strict transform of $\tilde{\Gamma} \times \tilde{\Xi}$. 
\end{proof}

\begin{proof}[Proof of Proposition \ref{prop relation}]
Using the moving lemma, we may assume that $\gamma_1$ and $\gamma_2$ do not meet each other. This allows us to apply Lemma \ref{lem existence Gamma} to the situation of $\tilde{\Gamma} = \gamma_1$ and $\tilde{\Sigma} = \gamma_2$. Thus we get a 2-cycle $\Gamma'$ on $X^{[2]}$ as in Lemma \ref{lem existence Gamma} (ii). We apply $\pi_*\tilde{\Phi}_*((-)\cdot\tau^*\Gamma')$ to the equality in Lemma \ref{lem divisorial relation} and get
\begin{equation}\label{eq apply Gamma'}
\pi_*\tilde{\Phi}_* (E\cdot \tau^*\Gamma')  = - \pi_*\tilde{\Phi}_*(h_{Q}|_{\widetilde{X^{[2]}}}\cdot\tau^*\Gamma') - \pi_*\tilde{\Phi}_*\tau^*((2h\hat\otimes 1 - 3\delta)\cdot\Gamma').
\end{equation}
For the last term, we apply Lemma \ref{lem Gamma} (ii) and have
\begin{align*}
\pi_*\tilde{\Phi}_*\tau^*((2h\hat\otimes 1 - 3\delta)\cdot\Gamma') & = \pi_*(i')^*\tilde{f}^*\alpha'+ \pi_*\Psi_*((2h \hat\otimes 1 - 3\delta)\cdot\Gamma')\\
 & = i^*\tilde{\pi}_*\tilde{f}^*\alpha'+ 2 e_2\gamma_1 + 2 e_1 \gamma_2,
\end{align*}
where $\alpha' = \varphi_*((2h\hat\otimes 1 - 3\delta)\cdot\Gamma')$ is a 1-cycle on $G(2,d+2)$. Similarly, we also have
\begin{align*}
\pi_*\tilde{\Phi}_*(h_{Q}|_{\widetilde{X^{[2]}}}\cdot\tau^*\Gamma')&  = h\cdot \pi_*\tilde{\Phi}_*(\tau^*\Gamma')\\
& = h\cdot \pi_*(i')^*\tilde{f}^*\alpha'' - \pi_*\Psi_*\sigma^*\Gamma'\\
& = h\cdot i^*\tilde{\pi}_*\tilde{f}^*\alpha'',
\end{align*}
where $\alpha'' = \varphi_* \Gamma'$ is a 2-cycle on $G(2,d+2)$. We apply Lemma \ref{lem Gamma} (i) to the left hand side of \eqref{eq apply Gamma'} and get
\[
\pi_*\tilde{\Phi}_* (E\cdot \tau^*\Gamma') = q_*p^*\mathfrak{a}',\qquad \mathfrak{a}' = p'_*(i_2)^*\Gamma'.
\]
Combining the above equations, we get
\[
2e_2\gamma_1 + 2 e_1\gamma_2 + q_*p^*\mathfrak{a}' = - i^*\Big( \tilde{\pi}_*\tilde{f}^*(\alpha') +  h\cdot \tilde{\pi}_*\tilde{f}^*(\alpha'') \Big).
\]
The right hand side of this equation is a 2-cycle on $\PP^{d+1}_K$ restricted to $X$ and hence it must be a multiple of $h^{d-1}$. We track the definition of $\mathfrak{a}'$ and get
\[
\mathfrak{a}' = p'_*i_2^*\Gamma' = p'_*i_2^* (\gamma_1\hat\otimes \gamma_2) = p_*q^*\gamma_1\cdot p_*q^*\gamma_2.
\]
Then Proposition \ref{prop relation} (ii) follows easily after comparing the degree of each side.

The statement of Proposition \ref{prop relation} (iii) is proved similarly.

To prove Proposition \ref{prop relation} (i), we first reduce it to the case when $\gamma_1$ is represented by a single irreducible curve. Assume that (i) is true for $\gamma_1$ and $\gamma'_1$. Namely,
\begin{align*}
 (2e_1-3)\gamma_1 \, & \in \, q_*p^*\CH_0(F) + \Z h^{d-1},\\
 (2e'_1 -3)\gamma'_1 \, & \in \, q_*p^*\CH_0(F) + \Z h^{d-1},
\end{align*}
where $e_1=\deg(\gamma_1)$ and $e'_1=\deg (\gamma'_1)$. Then clearly (i) is also true for $-\gamma_1$. Now we show that (i) also holds for $\gamma_1 + \gamma'_1$. By (ii), we know that
\[
 2e'_1\gamma_1 + 2e_1\gamma'_1\, \in \, q_*p^*\CH_0(F) + \Z h^{d-1} .
\]
Thus we get
\[
(2(e_1+e'_1) -3)(\gamma_1 + \gamma'_1) = (2e_1-3)\gamma_1 + (2e'_1-3)\gamma'_1 + (2e'_1\gamma_1 + 2e_1\gamma'_1) \, \in \, q_*p^*\CH_0(F) + \Z h^{d-1} .
\]

Now we assume that $\gamma_1$ is represented by a single irreducible curve $\tilde{\Gamma}\subset X$.
Thus we get a 2-cycle $\Gamma$ on $X^{[2]}$ as in \textit{(i)} of Lemma \ref{lem existence Gamma}. We apply $\pi_*\tilde{\Phi}_*((-)\cdot\tau^*\Gamma)$ to the equality in Lemma \ref{lem divisorial relation} and get
\begin{equation}\label{eq apply Gamma}
\pi_*\tilde{\Phi}_* (E\cdot \tau^*\Gamma)  = - \pi_*\tilde{\Phi}_*(h_{Q}|_{\widetilde{X^{[2]}}}\cdot\tau^*\Gamma) - \pi_*\tilde{\Phi}_*\tau^*((2h\hat\otimes 1 - 3\delta)\cdot\Gamma)
\end{equation}
Then we apply Lemma \ref{lem Gamma} (ii) to the last term and get
\begin{align*}
\pi_*\tilde{\Phi}_*\tau^*((2h\hat\otimes 1 - 3\delta)\cdot\Gamma) & = i^*\tilde{\pi}_*\tilde{f}^*\beta' - \pi_*\Psi_*\sigma^*((2h\hat\otimes 1 - 3\delta)\cdot\Gamma) \\
& =  i^*\tilde{\pi}_*\tilde{f}^*\beta' - (2e_1 -3)\gamma_1,
\end{align*}
where $\beta' = \varphi_*((2h\hat\otimes 1 - 3\delta)\cdot\Gamma )$. The other term on the right hand side of \eqref{eq apply Gamma} can be dealt with by
\[
 \pi_*\tilde{\Phi}_*(h_{Q}|_{\widetilde{X^{[2]}}}\cdot\tau^*\Gamma) = h\cdot i^*\tilde{\pi}_*\tilde{F}^*\beta'' - h\cdot \pi_*\Psi_*\sigma^*\Gamma = h\cdot i^*\tilde{\pi}_*\tilde{F}^*\beta'',
\]
where $\beta'' = \varphi_*\Gamma$. Apply Lemma \ref{lem Gamma} (i) to the left hand side of \eqref{eq apply Gamma}, we get
\[
\pi_*\tilde{\Phi}_* (E\cdot \tau^*\Gamma)  = q_*p^*\mathfrak{a},\qquad \mathfrak{a} = p'_*i_2^*\Gamma\in \CH_0(F).
\]
The above three displayed equations conbined with equation \eqref{eq apply Gamma}, we get
\[
 (2e_1 -3)\gamma_1 \, \in \, q_*p^*\CH_0(F) + \Z h^{d-1}.
\]
This finshes the proof.
\end{proof}

\section{Cubic hypersurfaces of small dimensions}

This section is devoted to applications of the universal generation result of the previous section. We relate the rationality problem of a cubic hypersurface of small dimenion to the geometry of its variety of lines. 

\subsection{A special decomposition of the diagonal}
In this subsection we fix a smooth cubic hypersurface $X\subset \PP^{d+1}_{\C}$ of dimension $d=3$ or $4$. Let $F=F(X)$ be the variety of lines on $X$ and $P\subset F\times X$ the universal line. Let $h\in \CH^1(X)$ be the class of a hyperplane section.

\begin{thm}\label{thm cubic 34}
Assume that $X$ admits a Chow theoretical decomposition of the diagonal. Then the following holds.
\begin{enumerate}[(i)]
\item If $d=3$, then there exists a symmetric 1-cycle $\theta \in \CH_1(F\times F)$ such that
\[
\Delta_X = x\times X + X \times x + (\gamma\times h + h\times \gamma) + (P\times P)_*\theta,\quad \text{in } \CH^3(X\times X),
\]
where $\gamma\in \CH_1(X)$. 

\item If $d=4$, then there exists a symmetric 2-cycle $\theta\in \CH_2(F\times F)$ such that
\[
 \Delta_X = x\times X + X \times x + \Sigma + (P\times P)_*\theta,\quad \text{in } \CH^4(X\times X),
 \] 
 where $\Sigma \in \CH_2(X)\otimes\CH_2(X)$ is a symmetric decomposable cycle. Moreover, $\Sigma$ can be chosen to be zero if $\mathrm{Hdg}^4(X)=\Z h^2$.
\end{enumerate}
\end{thm}

\begin{proof}
We will see that the construction of $\theta$ as in Theorem \ref{thm theta} suffices in the case of $d=3$. However, it is insufficient in the case of $d=4$ since it produces correspondences factoring through curves. Using an argument of Voisin \cite{voisin universal}, we can show that these problematic terms can be absorbed into the terms $(P\times P)_*\theta$ and $\Sigma$. We make it more precise.

Assume that $d=3$. Then by Proposition \ref{prop voisin}, there exist curves $Z_i$, correspondences $\Gamma_i\in \CH^2(Z_i\times X)$ and integers $n_i$ such that
\[
 \Delta_X = x\times X + X\times x +\sum n_i \Gamma_i\circ {}^t\Gamma_i, \quad \text{in } \CH^3(X\times X).
\]
As in the proof of Theorem \ref{thm theta}, we use the universal generation of $\CH_1(X)$ by lines to get $T_i\in \CH^2(Z_i\times F)$ such that $\Gamma'_i:= P\circ T_i\in \CH^2(Z_i\times X)$ agrees with $\Gamma_i$ over the generic point of $Z_i$. It follows that $\Gamma_i\circ {}^t\Gamma_i - \Gamma'_i\circ {}^t\Gamma'_i$ is a decomposable cycle, \textit{i.e.} of the form $\gamma_i\otimes h + h\otimes \gamma_i$ for some $\gamma_i\in \CH_1(X)$. Indeed $\Gamma_i = \Gamma'_i+ \Delta_i$ where $\Delta_i$ factors through points. Thus $\Gamma_i\circ {}^t\Gamma_i - \Gamma'_i\circ {}^t\Gamma'_i$ also factors through points and hence is decomposible. As a consequence, we have
\begin{align*}
 \Delta_X &= x\times X + X\times x + \sum n_i \Gamma_i\circ {}^t\Gamma_i\\
  & = x\times X + X\times x + \sum n_i \Gamma'_i\circ {}^t\Gamma'_i + \sum n_i (\gamma_i\otimes h + h\otimes \gamma_i)\\
  & = x\times X + X\times x + (P\times P)_*\left(\sum n_i T_i\circ {}^tT_i\right) + \gamma\otimes h + h\otimes \gamma, \quad \gamma = \sum n_i\gamma_i.
\end{align*}
Statement (i) follows by taking $\theta = \sum n_i T_i\circ {}^tT_i$.

Assume that $d=4$, then we can obtain surfaces $Z_i$ and correspondence $T_i\in\CH_3(Z_i\times X)$ similarly. Then we conclude that the cohomology class of
\[
 \Delta_X - x\times X - X\times x - (P\times P)_*\theta 
\]
is decomposable. Since the integral Hodge conjecture holds for $X$ (see Voisin \cite{voisin 07}), we know that
\begin{equation}\label{eq Gamma in coh}
 \Gamma:=\Delta_X - x\times X - X\times x - (P\times P)_*\theta - \gamma\otimes h - h\otimes\gamma -\Sigma = 0, \quad \text{in }\HH^4(X\times X,\Z),
\end{equation}
for some $\gamma\in \CH_1(X)$ and some symmetric cycle $\Sigma\in \CH_2(X)\otimes\CH_2(X)$. Since $\CH_2(F)\rightarrow \CH^1(X)$ is surjective and there exists $\tau\in \CH_2(F)$ such that $P_*\tau = h$, the term $\gamma\otimes h + h\otimes \gamma$ can be absorbed into the term $(P\times P)_*\theta$ and hence we can assume that $\gamma=0$ in equation \eqref{eq Gamma in coh}. There also exists $\tau'\in \CH_1(F)$ such that $P_*\tau' = h^2$ in cohomology. Thus we can assume that $\Sigma$ does not appear in equation \eqref{eq Gamma in coh} if $\mathrm{Hdg}^4(X)=\Z h^2$. By Proposition \ref{prop alg triv} (ii) below, we can modify $\theta$ by a homologically trivial cycle supported on the diagonal of $F\times F$ and assume that $\Gamma$ is algebraically trivial. Now by \cite{v nil, voisin nil}, we know that $\Gamma^{\circ N} = 0$ in $\CH^4(X\times X)$ for some sufficiently large integer $N$. In the expansion of this equation, any term involving $\Sigma$ is again decomposable of the same form and any power of $(P\times P)_*\theta$ is again of this form. After modifying $\theta$ and $\Sigma$, we get the equation
\[
\Delta_X = x\times X + X \times x + \Sigma + (P\times P)_*\theta,\quad \text{in } \CH^4(X\times X).
\]
This finishes the proof.
\end{proof}

\begin{rmk}
When $d=3$, the term $\gamma\otimes h + h\otimes \gamma$ can not be absorbed into $\theta$. This is because the homomorphism $q_*p^*: \HH^2(F,\Z)\rightarrow \HH^2(X,\Z)$ is not surjective. Given any 0-cycle $\mathfrak{o}_F\in \CH_0(F)$ of degree 1, we set $\mathfrak{l}= q_*p^*\mathfrak{o}_F\in \CH_1(X)$. Then the cycle $\gamma$ can always be chosen to be a multiple of $\mathfrak{l}$. This can be seen as follows. For any $\gamma'\in\CH_1(X)$ with $\deg(\gamma')=0$, then there exists $\gamma''\in \CH_1(X)$ such that $\gamma'= 5\gamma''$. Thus $\gamma'\otimes h = \gamma''\otimes 5h$ is contained in $(P\times P)_*\Big( \CH_0(F)\otimes \CH_1(F) \Big)$. Thus $\gamma\otimes h - \deg(\gamma) \mathfrak{l}\otimes h$ can be absorbed into $(P\times P)_*\theta$.

When $d=4$, it is known that $q_*p^*: \HH^6(F,\Z)\rightarrow \HH^4(X,\Z)$ is an isomorphism. If we assume that the integral Hodge conjecture holds for 1-cycles on $F$, then the term $\Sigma$ can always be absorbed into $\theta$. Note that this assumption has recently been established by Mongardi--Ottem \cite{mo}.
\end{rmk}

\begin{prop}\label{prop alg triv}
Let $X$ be a smooth cubic hypersurface of dimension $d=3$ or $4$. Let $\Gamma$ be a symmetric cycle on $X\times X$ such that $[\Delta_X]=[\Gamma]$ in $\HH^d(X\times X,\Z)$.
\begin{enumerate}[(i)]
\item If $d=3$, then 
\[
 \Delta_X - \Gamma = 0 , \quad \text{in }\CH_d(X\times X)/\mathrm{alg}.
\]
\item If $d=4$, then there exists a homologically trivial cycle $\sum_i a_iS_i \in \CH_2(F)$ such that
\[
 \Delta_X-\Gamma-\sum a_i P_i\times_{S_i} P_i = 0,\quad \text{in }\CH_4(X\times X)/\mathrm{alg},
\]
where $P_i = P|_{S_i}$ and $P_i\times_{S_i} P_i$ is understood to be its image in $X\times X$.
\end{enumerate}
\end{prop}

\begin{proof}
Recall that $h\in \CH^1(X)$ is the class of a hyperplane section. Let $l\subset X$ be a line. We note that, by Totaro \cite{totaro hilbert}, all cohomology groups involved in the proof are torsion free and hence $\HH^*(-)$ should be understood to be $\HH^*(-,\Z)$. By \cite[Corollary 2.4]{voisin universal}, there exists a $d$-cycle $\Gamma'$ on $X^{[2]}$ such that
\[
\mu^* \Gamma' = \Delta_X - \Gamma
\]
as algebraic cycles, where $\mu \in \CH_{2d}((X\times X)\times X^{[2]})$ is the correspondence defined by the closure of the graph of the rational map $X^2\dashrightarrow X^{[2]}$. As is explained in the proof of \cite[Proposition 2.6]{voisin universal}, we can require that
\[
 [\Gamma'] = 0,\quad \text{in }\HH^{2d}(X^{[2]}).
\]
If $d=3$, then Lemma \ref{lem alg equivalence} (i) applies. If $d=4$, then by Lemma \ref{lem alg equivalence} (ii) we can find homologically trivial $\sum_i a_iS_i\in \CH_2(F)$, such that $\Delta_X-\Gamma-\sum_i a_i P_i\times_{S_i} P_i$ is algebraically trivial. 
\end{proof}

\begin{lem}\label{lem alg equivalence}
Let $\Gamma'\in \CH^{d}(X^{[2]})_{\mathrm{hom}}$. 
\begin{enumerate}[(i)]
\item If $d=3$, then $\Gamma'$ is algebraically trivial.

\item If $d=4$, then there exist surfaces $S_i \subset F$ and integers $a_i$ such that $\sum a_i [S_i] =0 $ in $\HH^4(F,\Z)$ and
\[
 \mu^* \Gamma' = \sum a_i P_i \times_{S_i} P_i, \quad \text{in } \CH^{d}(X\times X)/\mathrm{alg},
\]
where $P_i$ is the universal line $P$ restricted to $S_i$.
\end{enumerate}
\end{lem}
\begin{proof}
This is a consequence of the explicit resolution $\tilde{\Phi}$ of the birational map $\Phi$ between $X^{[2]}$ and $P_X : =\PP(\mathcal{T}_{\PP^{d+1}}|_X)$; see the previous section and \cite[Proposition 2.9]{voisin universal}. Recall that we have the following commutative diagram.
\[
 \xymatrix{
E\ar[rd]^j \ar[dd]_{\pi_2}\ar[rr]^{\pi_1} & &P_1 \ar[d]^{i_1} \\
 & \widetilde{X^{[2]}}\ar[d]^\tau \ar[r]^{\tilde{\Phi}} & P_X \\
 P_2 \ar[r]^{i_2} & X^{[2]} &
}
\]
Here $P_2=P^{[2]}/F$ is the relative Hilbert scheme of two points on the universal line $P/F$ and $P_1=P$. The morphism $\tau$ is the blow up of $X^{[2]}$ along $P_2$ and $\tilde{\Phi}$ is the blow up of $P_X$ along $P_1$. The two blow up morphisms share the same exceptional divisor
\[
 E = P_1\times_F  P_2.
\]
Note that $\eta_1:P_1\rightarrow F$ is a $\PP^1$-bundle over $F$ and $\eta_2: P_2 \rightarrow F$ is a $\PP^2$-bundle over $F$. Let $\xi_i\in\CH^1(P_i)$, $i=1,2$, be the first Chern classes of the relative $\calO(1)$-bundles. By abuse of notation, we still use $\xi_i$ to denote its pull back to $E$.
By the blow up formula, we know that
\[
\CH_d(\widetilde{X^{[2]}}) = \tilde{\Phi}^*\CH_d(P_X) \oplus j_* \pi_1^* \CH_{d-2}(P_1) \oplus j_*\Big(\xi_2\cdot  \pi_1^*\CH_{d-1}(P_1) \Big).
\]
Thus
\begin{equation}\label{eq sigma pull back}
 \tau^*\Gamma' =  \tilde{\Phi}^*\Gamma'_0 + j_*\pi_1^*\Gamma'_1 + j_*(\xi_2 \pi_1^* \Gamma'_2)
\end{equation}
where
\[
\Gamma'_0 \in \CH_d(P_X)_{\mathrm{hom}} , \quad \Gamma'_1 \in \CH_{d-2}(P_1)_{\mathrm{hom}}, \qquad \Gamma'_2 \in \CH_{d-1}(P_1)_{\mathrm{hom}}.
\]
By the projective bundle formula, we have
\begin{align*}
 \Gamma'_1 & = \eta_1^* \Gamma'_{1,0} + \xi_1\cdot \eta_1^*\Gamma'_{1,1},\quad \Gamma'_{1,0}\in \CH_{d-3}(F)_{\mathrm{hom}} \text{ and }\Gamma'_{1,1}\in \CH_{d-2}(F)_{\mathrm{hom}};\\
 \Gamma'_2 & = \eta_1^* \Gamma'_{2,0} + \xi_1\cdot \eta_1^*\Gamma'_{2,1},\quad \Gamma'_{2,0}\in \CH_{d-2}(F)_{\mathrm{hom}} \text{ and }\Gamma'_{2,1}\in \CH_{d-1}(F)_{\mathrm{hom}}.
\end{align*}
Note that 
\begin{align*}
 \tau_*j_*(\pi_1^*\eta_1^*\Gamma'_{1,0}) & = i_{2,*}\pi_{2,*}(\pi_1^*\eta_1^*\Gamma'_{1,0}) = 0, \\
 \tau_*j_*(\xi_2\cdot \pi_1^*\eta_1^*\Gamma'_{2,0}) & = i_{2,*}\pi_{2,*}(\xi_2\cdot \pi_1^*\eta_1^*\Gamma'_{2,0}) = 0.
\end{align*}
Applying $\sigma_*$ to equation \eqref{eq sigma pull back}, we get
\begin{align*}
\Gamma' & = \Phi^*\Gamma'_0 + i_{2,*}\pi_{2,*}\pi_1^*(\xi_1\cdot \eta_1^*\Gamma'_{1,1}) + i_{2,*}\pi_{2,*}\Big( \xi_2\cdot \pi_1^*(\xi_1\cdot \eta_1^*\Gamma'_{2,1}) \Big)\\
 & =  \Phi^*\Gamma'_0 + i_{2,*}\eta_2^*\Gamma'_{1,1} + i_{2,*}\Big( \xi_2\cdot \eta_2^*\Gamma'_{2,1}\Big).
\end{align*}

Now we follow the argument in the proof of \cite[Theorem 1.1]{voisin universal}. The first fact is that the algebraic equivalence is the same as the homological equivalence on $P_X$. Thus we see that $\Gamma'_0$ is algebraically equivalent to zero. When $d=3$ or $4$, the cycle $\Gamma'_{2,1}$ is either of codimension 0 or of codimension 1. Thus we always have that $\Gamma'_{2,1}$ is algebraically equivalent to zero. When $d=3$ we have $\dim F =2$ and $\Gamma'_{1,1}\in \CH^1(F)_{\mathrm{hom}}$. So $\Gamma'_{1,1}$ is also algebraically equivalent to zero in this case. Statement (1) follows.

Assume $d=4$ and hence $\dim F =4$. Thus $\Gamma'_{1,1}\in\CH_2(F)_{\mathrm{hom}}$. We can write
\[
 \Gamma'_{1,1} = \sum_i a_i S_i
\]
where $S_i\subset F$ are surfaces. Then an explicit computation gives
\[
\mu^* (i_{2,*}\eta_2^* S_i) = P_i \times_{S_i} P_i
\]
as cycles. Hence the lemma is proved.
\end{proof}

\subsection{Algebraicity of the Beauville--Bogomolov form}

Let $X$ be a smooth cubic fourfold and let $F$ be its variety of lines. It is known that $F$ is a hyperk\"ahler variety. By Beauville--Donagi \cite{bd}, we know that $\alpha \mapsto \hat\alpha:=p_*q^*\alpha$ gives an isomorphism between $\HH^4(X,\Z)$ and $\HH^2(F,\Z)$ and the Beauville--Bogomolov bilinear form on $\HH^2(F,\Z)$ is given by
\[
\mathfrak{B}(\hat\alpha,\hat\beta) = \langle\alpha, h^2 \rangle_X \langle \beta,h^2 \rangle_X - \langle \alpha, \beta \rangle_X,
\]
for all $\alpha,\beta \in \HH^4(X,\Z)$. Let $\alpha_i$, $i=1,\ldots,23$, be an integral basis of $\HH^4(X,\Z)$. Then $\{\hat\alpha_i\}$ form an integral basis of $\HH^2(F,\Z)$ and let $\{ \hat\alpha_i^\vee \}$ be the dual basis of $\HH^6(F,\Z)$. Then the Beauville--Bogomolov form corresponds to the canonical integral Hodge class
\[
 q_{\mathfrak{B}} = \sum _{i,j=1}^{23} b_{ij}\hat\alpha_i^\vee \otimes \hat\alpha_j^\vee \in \HH^{12}(F\times F,\Z),
\]
Where $b_{ij}=\mathfrak{B}(\hat\alpha_i,\hat\alpha_j)$.

\begin{prop}
Let $X$ be a smooth cubic fourfold and $F$ be its variety of lines as above. If $q_{\mathfrak{B}}\in \HH^{12}(F\times F,\Z)$ is algebraic then $X$ is universally $\CH_0$-trivial. The converse is true if the integral Hodge conjecture holds for 1-cycles on $F$ (\textit{e.g.} if $\mathrm{Hdg}^4(X)$ is generated by $h^2$).\footnote{The integral Hodge conjecture for $1$-cycles on $F$ has recently been established by Mongardi--Ottem \cite{mo}.} 
\end{prop}

\begin{proof}
Assume that $q_{\mathfrak{B}}$ is algebraic. Since $[\Delta_X] + (P\times P)_*q_{\mathfrak{B}}$ pairs to zero with $\alpha \otimes \beta$ for all $\alpha,\beta\in \HH^4(X,\Z)_{\mathrm{prim}}$, we know that $[\Delta_X] + (P\times P)_*q_{\mathfrak{B}}$ is decomposable. Thus $X$ admits a cohomological decomposition of the diagonal, which implies universal $\CH_0$-triviality by Voisin \cite{voisin universal}. Assume that $\CH_0(X)$ is universally trivial. Then we get the cycle $\theta\in \CH_2(F\times F)$ as in Theorem \ref{thm cubic 34}. The Hodge class $q_{\mathfrak{B}}+[\theta]$ is decomposable and hence algebraic by the assumption that the integral Hodge conjecture holds for 1-cycles on $F$.
\end{proof}

\subsection{The minimal class on the intermediate Jacobian of a cubic threefold}
Let $X$ be a smooth cubic threefold and let $F$ be the surface of lines on $X$. Let $J^3(X)$ be the intermediate Jacobian of $X$. It is known that the Abel--Jacobi map
\[
\phi: F\rightarrow J^3(X)
\]
induces an isomorphism 
\[
\phi^*: \HH^1(J^3(X),\Z) \longrightarrow \HH^1(F,\Z).
\]
There is a natural identification $\HH^1(J^3(X),\Z)\cong \HH^3(X,\Z)$, under which we have
\begin{equation}\label{eq phi aj}
\phi^*\alpha = \hat\alpha =p_*q^*\alpha,
\end{equation}
for all $\alpha\in \HH^1(J^3(X), \Z)=\HH^3(X,\Z)$. By taking the difference or the sum, we have the following morphisms
\begin{align*}
\phi_+ : &F\times F \longrightarrow J^3(X),\quad (u,v)\mapsto \phi(x) + \phi(y),\\
\phi_- : & F\times F \longrightarrow J^3(X), \quad (u,v)\mapsto \phi(x) - \phi(y).
\end{align*}
By \cite[Theorem 13.4]{cg} and its proof, the image of $\phi_-$ is a theta divisor of $J^3(X)$ and $\phi_-$ has degree 6 onto its image. 

\begin{lem}
The image of $\phi_+$ is a divisor $D_+$ of cohomological class $3\Theta$ and $\phi_+$ has degree 2 onto its image.
\end{lem}
\begin{proof}
We first give a description of the generic behaviour of the degree 6 morphism $\phi_-$. Let $l_1$ and $l_2$ be a general pair of lines on $X$. Then $l_1$ and $l_2$ span a linear $\PP^3$ and its intersection with $X$ is a smooth cubic surface $\Sigma$ containing $l_1$ and $l_2$. Realise $\Sigma$ as a blow up of $\PP^2$ in 6 points $\{P_1,P_2,\ldots,P_6\}$ with $E_i$, $i=1,2,\ldots,6$, being the exceptional curves. This can be chosen in such a way that $E_1=l_1$ and $E_2=l_2$. Let $C_i\subset \Sigma$ be the strict transform of the conic in $\PP^2$ that passes all the 6 points except $P_i$. Let $L_{ij}$, $1\leq i<j\leq 6$, be the strict transform of the line in $\PP^2$ passing through the points $P_i$ and $P_j$. Then $E_i$, $C_j$ and $L_{i'j'}$ are the 27 lines on $\Sigma$. It is clear that, in $\CH^1(\Sigma)$, we have
\[
E_1-E_2 = C_1-C_2 = L_{23}-L_{13} = L_{24}-L_{14} = L_{25}- L_{15} = L_{26} - L_{16}.
\]
It follows that $(\phi_{-})^{-1} \phi_{-}([l_1],[l_2])$ consists of the following 6 points
\begin{equation}\label{eq list}
 ([E_1], [E_2]), \quad ([C_1],[C_2]),\quad ([L_{23}], [L_{13}]), \quad ([L_{24}],[L_{14}]), \quad ([L_{25}], [L_{15}]),\quad ([L_{26}], [L_{16}]).
\end{equation}

Now let $l'_1$ and $l'_2$ another general pair of lines on $X$. Assume that $l''_1$ and $l''_2$ be a different pair of lines such that $\phi_{+}([l'_1],[l'_2]) = \phi_{+}([l''_1],[l''_2])$. Then we have $\phi_{-}([l'_1],[l''_2]) = \phi_{-}([l''_1],[l'_2])$ and hence $([l'_1],[l''_2])$ and $([l''_1],[l'_2])$ form two distinct points from the list \eqref{eq list}. This implies that $l'_1$ meets $l'_2$ (and that $l''_1$ meets $l''_2$) which is a contradiction since $l'_1$ and $l'_2$ are general. This implies that $\phi_+$ is of degree 2 onto its image.

The lemma follows from the fact that $(\phi_{+})_*[F\times F] = (\phi_-)_*[F\times F] = 6\Theta$ in $\HH^2(J^3(X),\Z)$. This fact can be seen as follows. It is known that $\phi_*[F] = \frac{\Theta^3}{3!}$ in $\HH^6(J^3(X),\Z)$. It follows that $[-1]_*\phi_*[F] = \phi_*[F]$ in $\HH^6(J^{3}(X),\Z)$. Then we have
\begin{align*}
(\phi_{-})_*[F\times F] & = (\mu_{+})_*(Id\times [-1])_*(\phi\times \phi)_*[F\times F] \\
& = (\mu_+)_*\Big( \phi_*[F]\otimes [-1]_*\phi_*[F] \Big)\\
& = (\mu_+)_*( \phi_*[F]\otimes \phi_*[F] )\\
&= (\mu_+)_*[F\times F],
\end{align*}
where $\mu_+:J^3(X)\times J^3(X) \rightarrow J^3(X)$ is the addition morphism.
\end{proof}

\begin{prop}
If $\CH_0(X)$ is universally trivial, then the following holds.
\begin{enumerate}[(i)]
\item The minimal class of $J^3(X)$ is algebraic and supported on the divisor $D_+\subset J^3(X)$ of cohomology class $3\Theta$.
\item Twice of the minimal class is represented by a symmetric 1-cycle supported on a theta divisor of $J^3(X)$.
\end{enumerate}
\end{prop}

\begin{proof}
The morphism $\phi_+$ factors as 
\[
\xymatrix{
 F\times F\ar[r]^{(\phi,\phi)\qquad} & J^3(X)\times J^3(X)\ar[r]^{\qquad\mu_+} &J^3(X)
}
\]
where $\mu_+(x,y)=x+y$ is the summation morphism. Thus for any $\alpha\in \HH^1(J^3(X),\Z)$, we have
\[
(\phi_+)^*\alpha = (\phi,\phi)^*(\mu_+)^*\alpha = (\phi,\phi)^*(\alpha\otimes 1 + 1\otimes \alpha) = \phi^*\alpha \otimes 1 + 1\otimes \phi^*\alpha.
\]
Similarly, we also have $(\phi_-)^*\alpha = \phi^*\alpha\otimes 1 - 1\otimes \phi^*\alpha$. Let $\theta\in \CH_1(F\times F)$ be the symmetric cycle as in Theorem \ref{thm cubic 34}. Then for all $\alpha,\beta\in \HH^1(J^3(X),\Z)$ we have
\begin{align*}
(\phi_+)_*[\theta ]\cup \alpha\cup\beta & =[\theta]\cup (\phi_+)^*\alpha \cup (\phi_+)^*\beta\\
 & = [\theta]\cup (\phi^*\alpha \otimes 1 + 1 \otimes \phi^*\alpha)\cup (\phi^*\beta \otimes 1 + 1\otimes\beta)\\
 & = [\theta]\cup ((\hat\alpha\cup\hat\beta)\otimes 1 + \hat\alpha \otimes \hat\beta - \hat\beta \otimes \hat\alpha + 1\otimes (\hat\alpha \cup \hat\beta) )\\
 & = 2\phi_*[\theta_1] \cup \alpha\cup\beta + 2 \langle \alpha,\beta \rangle_X ,
\end{align*}
where $\theta_1=(pr_1)_*\theta \in \CH_1(F)$. The same computation shows that
\[
 (\phi_+)_*[\theta_1\otimes \mathfrak{o}]\cup \alpha\cup \beta =  (\phi_+)_*[\mathfrak{o}\otimes \theta_1]\cup \alpha\cup \beta = \phi_*[\theta_1]\cup \alpha \cup \beta,
\]
where $\mathfrak{o}\in J^3(X)$ is the zero element. Take $\tilde\theta = \theta - \theta_1\otimes \mathfrak{o} - \mathfrak{o}\otimes \theta_1$, then
\[
 (\phi_+)_*[\tilde\theta] \cup \alpha\cup \beta = 2 \langle \alpha, \beta\rangle_X.
\]
Since $\tilde\theta$ is again a symmetric cycle, we know that $(\phi_+)_*\tilde\theta = 2 \eta$ for some $\eta\in \CH_1(J^3(X))$ supported on $D_+$. Thus $[\eta]\cup\alpha \cup \beta = \langle \alpha, \beta \rangle_X$ and hence $-[\eta]$ is the minimal class on $J^3(X)$. We carry out the same computation for $\phi_-$ and see that 
\[
(\phi_-)_*[\tilde\theta] \cup \alpha \cup \beta = -2\langle \alpha,\beta \rangle_X.
\]
Thus the cohomology class of $(\phi_-)_*\tilde\theta$ is twice the minimal class. Furthermore it is symmetric (with respect to multiplication by $-1$ on $J^3(X)$) and supported on the image of $\phi_-$ which is a theta divisor. 
\end{proof}

\subsection{Cubic fivefolds}
Let $X\subset \PP^6$ be a cubic fivefold and let $F$ be its variety of lines. Let $J^5(X)$ be the intermediate Jacobian of $X$, which happens to a be a principally polarized abelian variety.

\begin{prop}
If both $\CH_0(X)$ and $\CH_0(F)$ are universally trivial, then there exist curves finitely many $C_i$ together with a splitting surjective homomorphism $\bigoplus_i J(C_i)\longrightarrow J^5(X)$. 
\end{prop}

\begin{proof}
Since $\CH_0(F)\rightarrow \CH_1(X)$ is universally surjective, we know that $\CH_1(X)$ is universally trivial. By Corollary \ref{cor voisin generalized}, the universal triviality of both $\CH_0(X)$ and $\CH_1(X)$ implies the existence of curves $C_i$, corresponcences $\Gamma_i\in\CH^3(C_i\times X)$ and integers $n_i$ such that
\begin{equation}\label{eq intersection Ci}
\sum n_i \langle \Gamma_i^*\alpha,\sigma_i^*\Gamma_i^*\beta \rangle_{C_i} = \langle \alpha,\beta \rangle_X,
\end{equation}
for all $\alpha,\beta \in F^5\HH^5(X,\Z) = \HH^5(X,\Z)$, where $\sigma_i:C_i\rightarrow C_i$ is either the identity map or an involution. Each cycle $\Gamma_i$ defines the associated Abel--Jacobi map
\[
\phi_i : J(C_i)\longrightarrow J^5(X).
\]
Combining them together we have the surjective homomorphism
\[
\phi: \bigoplus_{i}J(C_i)\longrightarrow J^5(X).
\]
Then the equation \eqref{eq intersection Ci} implies that $\phi$ has a section given by $\sum \sigma_i^\vee\circ \phi_i^\vee$.
\end{proof}


\begin{thebibliography}{CU}
\bibitem{ahtv} N.~Addington, B.~Hassett, Y.~Tschinkel and A.~V\'arilly-Alvarado, \textit{Cubic fourfolds fibered in sextic del Pezzo surfaces}, preprint 2016, arXiv:1606.05321.

\bibitem{ak} A.B.~Altman and S.L.~Kleiman, \textit{Foundations of the theory of Fano schemes}, Compos. Math. \textbf{34} (1), 1977, 3--47.

\bibitem{am} M. Artin and D.~Mumford, \textit{Some elementary examples of unirational varieties which are not rational}, Proc. London Math. Soc. (3) \textbf{25} (1972), 75--95.

\bibitem{bd}A.~Beauville and Donagi, \textit{La vari\'et\'e des droites d'une hypersurface cubique de dimension 4}. C. R. Acad. Sci. Paris S\'er. I Math. \textbf{301} (1985), no. 14, 703--706.

\bibitem{bs} S.~Bloch and V.~Srinivas, \textit{Remarks on correspondences and algebraic cycles}, Amer. J. of Math. \textbf{105} (1983), 1235--1253.

\bibitem{cg} C.H.~Clemens and P.A.~Griffiths, \textit{Intermediate Jacobian of cubic fourfolds}, Ann. Math. \textbf{95} No. 2, 1972, 281--356.

\bibitem{ct-p} J.-L.~Colliot-Th\'el\`ene and A.~Pirutka, \textit{Hypersurfaces quartiques de dimension 3: nonrationalit\'e stable}, Annales Scientifiques de l'ENS \textbf{49}, fascicule 2 (2016), 371--397.

\bibitem{gs} S.~Galkin and E.~Shinder, \textit{The Fano variety of lines and rationality problem for a cubic hypersurface}, preprint, 2014.

\bibitem{hassett} B. Hassett, \textit{Special cubic fourfolds}, Compos. Math. \textbf{120}, 1--23, 2000.

\bibitem{hassett2} B. Hassett, \textit{Some rational cubic fourfolds}, J. Alg. Geom. \textbf{8} (1999), no. 1, 103--114.

\bibitem{hkt} B.~Hassett, A.~Kresch and Y.~Tschinkel, \textit{Stable rationality and conic bundles}, Math. Ann. \textbf{365} (2016), no. 3-4, 1201--1217.

\bibitem{hpt1} B.~Hassett, A.~Pirutka and Y.~Tschinkel, \textit{Stable rationality of quadric surface bundles over surfaces}, preprint 2016, arXiv: 1603.09262.

\bibitem{hpt2} B.~Hassett, A.~Pirutka and Y.~Tschinkel, \textit{A very general quartic double fourfold is not stably rational}, preprint 2016, arXiv: 1605.03220.

\bibitem{ht} B.~Hassett and Y.~Tschinkel, \textit{On stable rationality of Fano threefolds and del Pezzo fibrations}, J. Reine Angew. Math., to appear.

%\bibitem{hironaka} H.~Hironaka, \textit{Smoothing of algebraic cycles of small dimensions}, American %Journal of Mathematics, Vol. 90, No. 1 (1968), pp. 1--54.

\bibitem{im} V. A.~Iskovskikh and Yu.~Manin, \textit{Three-dimensional quartics and counterexamples to the Lu ̈roth problem}, Mat. Sb. (N.S.) \textbf{86}(128) (1971), 140--166.

\bibitem{kuz} A.~Kuznetsov, \textit{Derived categories of cubic fourfolds}, in ``Cohomological and geometric approaches to rationality problems", 219--243, Progr. Math. 282, Birkh\"auser Boston, Inc., Boston, MA, 2010.

\bibitem{mo} G.~Mongardi and J.~Ottem, \textit{Curve classes on irreducible holomorphic symplectic varieties}, preprint 2018, arXiv: 1805.12019.

\bibitem{okada} T.~Okada, \textit{Stable rationality of cyclic covers of projective spaces}, preprint 2016, arXiv: 1604.08417.

\bibitem{relation}M.~Shen, \textit{On relations among 1-cycles on cubic hypersurfaces}, J. Alg. Geom. \textbf{23} (2014), 539--569.

\bibitem{pt} M.~Shen, \textit{Prym--Tjurin construction on cubic hypersurfaces}, Doc. Math. \textbf{19} (2014), 867--903.

\bibitem{hk} M.~Shen, \textit{Hyperk\"ahler manifolds of Jacobian type}, J. reine Angew. Math., to appear.

\bibitem{shimada} I.~Shimada, \textit{On the cylinder isomorphism associated to the family of lines on a hypersurface}, J. Fac. Sci. Univ. Tokyo, Sect. IA, Math. \textbf{37} (1990), 703--719.

\bibitem{totaro} B.~Totaro, \textit{Hypersurfaces that are not stably rational}, J. Amer. Math. Soc., to appear.

\bibitem{totaro hilbert} B.~Totaro, \textit{The integral cohomology of the Hilbert scheme of two points}, preprint 2015.

\bibitem{v nil} V.~Voevodsky, \textit{A nilpotence theorem for cycles algebraically equivalent to zero}, Internat. Math. Res. Notices, No. 4 (1995), 187--198.

\bibitem{voisin 86} C.~Voisin, \textit{Th\'eor\`eme de Torelli pour les cubiques de $\PP^5$}, Invent. Math. \textbf{86} (1986), 577 -- 601.

\bibitem{voisin nil} C.~Voisin, \textit{Remarks on zero-cycles of self-products of varieties}, in ``Moduli of vector bundles" (Proceedings of the Taniguchi congress on vector bundles), Maruyama Ed., Decker (1994), 265--285.

\bibitem{voisin 07} C.~Voisin, \textit{Some aspects of the Hodge conjecture}, Jpn. J. Math. 2 (2007), no. 2, 261--296. 

\bibitem{voisin aj} C.~Voisin, \textit{Abel--Jacobi map, integral Hodge classes and decomposition of the diagonal}, J. Alg. Geom. \textbf{22} (2013), no. 1, 141--174.

\bibitem{voisin invent} C.~Voisin, \textit{Unirational threefolds with no universal codimension 2 cycle}.
Invent. Math. \textbf{201} (2015), no. 1, 207--237.

\bibitem{voisin universal} C.~Voisin, \textit{On the universal $\CH_0$ group of cubic hypersurfaces}, JEMS,  Vol. \textbf{19}, Issue 6 (2017), 1619--1653.

\end{thebibliography}
\end{document}